\documentclass[10pt]{article}
\usepackage{bbm}
\usepackage{amsmath}
\usepackage{amsthm}
\usepackage{tabularx}
\usepackage{enumitem}
\usepackage{amssymb} 
\usepackage{mathrsfs}
\usepackage{caption}
\usepackage{graphicx}
\usepackage{authblk}
\usepackage[hypertexnames=false,colorlinks=true,urlcolor=blue,linkcolor=blue,citecolor=blue]{hyperref}
\usepackage[numbers,comma,square,sort&compress]{natbib}
\usepackage[letterpaper,text={6.5in,9in},centering]{geometry}

\makeatletter\@addtoreset{equation}{section}\makeatother

\newtheorem{thm}{Theorem}[section] 
\newtheorem{lem}[thm]{Lemma}  
\newtheorem*{lem*}{Lemma}  
 
\newtheorem{prop}[thm]{Proposition}  
\newtheorem{hyp}[thm]{Hypothesis}
\newtheorem{defn}[thm]{Definition}
\newtheorem{rmk}[thm]{Remark}

\newenvironment{Acknowledgment}%
 {\begin{trivlist}\item[]\textbf{Acknowledgments.}}{\end{trivlist}}
\usepackage{tikz}

\graphicspath{{Figures/}}

\begin{document}

\title{The role of spatial dimension in the emergence of localised radial patterns from a Turing instability}
\author{Dan J. Hill}

\affil{\small Fachrichtung Mathematik, Universit\"at des Saarlandes, Postfach 151150, 66041 Saarbr\"ucken, Germany}

\date{}
\maketitle

\begin{abstract}
\noindent The emergence of localised radial patterns from a Turing instability has been well studied in two and three dimensional settings and predicted for higher spatial dimensions. We prove the existence of localised $(n+1)$-dimensional radial patterns in general two-component reaction-diffusion systems near a Turing instability, where $n>0$ is taken to be a continuous parameter. We determine explicit dependence of each pattern's radial profile on the dimension $n$ through the introduction of $(n+1)$-dimensional Bessel functions, revealing a deep connection between the formation of localised radial patterns in different spatial dimensions.
\end{abstract}
%
%
%
%
%
%
%
%
%
%
\section{Introduction}\label{s:intro}

In this work, we study the effect of a spatial domain's dimension on the existence of stationary localised radial patterns bifurcating from a Turing instability. We consider a general class of two-component reaction-diffusion systems which, for a spatial coordinate $\mathbf{x}:=(x_1, \dots, x_{n+1})\in\mathbb{R}^{n+1}$ with $n\in\mathbb{N}_{0}$, can be written as
\begin{equation}\label{R-DEqn:Gen}
	\partial_{t}\mathbf{u} = \mathbf{D}(\mu)\Delta\mathbf{u} - \mathbf{f}(\mathbf{u},\mu),
\end{equation}
where $\mathbf{u}\in\mathbb{R}^{2}$, $\Delta:=\partial_{x_{1}}^{2} + \dots + \partial_{x_{n+1}}^{2}$ is the $(n+1)$-dimensional Laplace operator, $\mathbf{D}$ is the diffusion coefficient matrix, $\mathbf{f}$ is a nonlinear function, and $\mu\in\mathbb{R}$ is the bifurcation parameter. Reaction-diffusion systems of the form \eqref{R-DEqn:Gen} have been found to model numerous physical phenomena, including the growth of desert vegetation \cite{vonHardenberg2001,Zelnik2015Gradual}, the spread of epidemics \cite{Allen2008RDEpidemic}, and nonlinear optics \cite{LugiatoLefever1987Optics}. Beyond specific applications, there is also interest in considering Turing instabilities in $n$ spatial dimensions \cite{VanGorder2021Turing_Manifolds,Alber2005n-dimensional}. Fully localised patterns, consisting of compact patches of spatially-oscillating structure surrounded by a uniform state, have been numerically found to emerge as solutions of \eqref{R-DEqn:Gen} when the uniform state undergoes a Turing instability \cite{lloyd2008localized,Clerc2021labyrinthine}. However, our mathematical understanding of the emergence of such patterns remains quite limited when the spatial dimension is larger than one; see \cite{bramburger2024localized} for a recent review on the mathematical study of localised patterns.

A notable subclass of fully localised patterns are those which depend purely on the radial coordinate $r:=|\mathbf{x}|$. We refer to such solutions as \textit{radial patterns}, but they are also called \textit{axisymmetric} or \textit{spherically-symmetric} patterns in two and three-dimensions, respectively. Such two-dimensional patterns have been found in vegetation patches and rings \cite{Meron2007,Hill2022LocalizedVegetation}, chemical reactions \cite{vanag2004stationary}, barrier gas discharge \cite{Gurevich2003Ring}, laser processing \cite{Katayama2003Rings}, and as Kerr solitons \cite{mcsloy2002computationally,menesguen2006optical}; see \cite{Clerc2023Rings} and the references therein for more examples of radial patterns in physical contexts. Three dimensional localised radial patterns are less well known, but have been observed in Belousov--Zhabotinsky reactions \cite{Bansagi2011Tomography} and as cavity bullets in nonlinear optics \cite{Tlidi2021LightBullets}. When restricting to stationary radial solutions, \eqref{R-DEqn:Gen} becomes
\begin{equation}\label{R-DEqn}
	\mathbf{0} = \mathbf{D}(\mu)\Delta_{n}\mathbf{u} - \mathbf{f}(\mathbf{u},\mu),
\end{equation}
where the $(n+1)$-dimensional radial Laplace operator takes the form $\Delta_{n}:=\partial^{2}_{r} + \tfrac{n}{r} \partial_{r}$.   To see this, note that $r=(x^2_1 + x^2_2 + \dots + x_{n+1}^2)^{\frac{1}{2}}$ and $\partial_{x_{i}}\mathbf{u}(r) = \frac{x_i}{r} \mathbf{u}'(r)$ for each $i=1,2,\dots,n+1$. The spatial dimension $n$ can then be treated as a parameter of \eqref{R-DEqn}, where one might expect solutions of \eqref{R-DEqn} to depend continuously on $n$. Of course, while we can consider $n\in\mathbb{R}$ and find localised solutions of \eqref{R-DEqn}, these solutions only correspond to localised radial solutions of \eqref{R-DEqn:Gen} in the case when $n\in\mathbb{N}$. However, the dependence of each solution on the spatial dimension $n$ provides new insight into the connection between two, three, and $n$-dimensional patterns that would be extremely difficult to deduce using Cartesian coordinates.

Our present work can be seen as the extension of two previous mathematical studies of localised radial patterns. In 2009, Lloyd and Sandstede \cite{lloyd2009localized} considered localised solutions the quadratic-cubic Swift--Hohenberg equation
\begin{equation}\label{eqn:SH}
    0 = -(1+\Delta_n)^2 u - \mu u + \nu u^2 - u^3,
\end{equation}
for $n=1$, where $\nu\in\mathbb{R}$ is a fixed parameter. Using the radial spatial dynamics theory of Scheel \cite{scheel2003radially}, the authors proved the existence of a localised spot $u(r) = u_{A}(r)$ when $\nu\neq0$---which we call spot A---and two localised ring patterns $u(r) = \pm u_{R}(r)$ when $|\nu|>\sqrt{\frac{27}{38}}$. These solutions have the asymptotic profile
\begin{equation*}
    u_{A}(r) = \frac{C_{A}}{\nu} \mu^{\frac{1}{2}}J_{0}(r) + \mathcal{O}(\mu), \qquad  u_{R}(r) = C_{R} \mu^{\frac{3}{4}} r J_{1}(r) + \mathcal{O}(\mu),
\end{equation*}
for $r\in[0,r_0]$, where $C_{A}, C_{R}, r_0>0$ are fixed constants and $J_{\ell}(r)$ is the $\ell^{\mathrm{th}}$ order Bessel function of the first kind. Finally, the authors proved that the spot A solution undergoes a fold bifurcation along a curve $\nu \sim \mu^{\frac{1}{4}}|\log(\mu)|^{\frac{1}{2}}$.

In 2013, McCalla and Sandstede \cite{mccalla2013spots} introduced novel geometric blow-up coordinates to extend the approach of \cite{lloyd2009localized} and prove the existence of a second spot solution $u(r) = u_{B}(r)$ when $|\nu|>\sqrt{\frac{27}{38}}$, which we call spot B. This solution has the asymptotic profile
\begin{equation*}
    u_{B}(r) = -\mathrm{sgn}(\nu) \frac{C_{B}}{|\nu|^{\frac{1}{2}}} \mu^{\frac{3}{8}}J_{0}(r) + \mathcal{O}(\mu^{\frac{1}{2}})
\end{equation*}
for $r\in[0,r_0]$, with fixed constant $C_{B}>0$. The authors then applied their approach to \eqref{eqn:SH} for $n=2$ in order to prove the existence of three-dimensional localised spot A and B solutions. In this case, the solutions have the asymptotic profile
\begin{equation*}
  u_{A}(r) = \frac{C_{A}}{\nu} \mu^{\frac{1}{2}} \frac{\sin(r)}{r} + \mathcal{O}(\mu),  \qquad u_{B}(r) = -\mathrm{sgn}(\nu) \frac{C_{B}}{|\nu|^{\frac{1}{2}}} \mu^{\frac{1}{4}}\frac{\sin(r)}{r} + \mathcal{O}(\mu^{\frac{1}{2}})
\end{equation*}
for $r\in[0,r_0]$, where we note that the fixed constants $C_{A},C_{B}>0$ are not necessarily the same as the $n=1$ case. A formal scaling argument led the authors to conjecture that spot B solutions of \eqref{eqn:SH} will possess a  scaling of $\mu^{\frac{4-n}{8}}$ for a general dimension $n>0$.\footnote{We note that there is a discrepancy between this asymptotic scaling and the one stated in \cite{mccalla2013spots}, as the authors therein consider a spatial domain $\mathbb{R}^{n}$ rather than $\mathbb{R}^{n+1}$.} While the existence of three-dimensional ring solutions to \eqref{eqn:SH} should also follow from the approach presented in \cite{mccalla2013spots}, an existence proof has never been completed, and so the existence of three-dimensional rings remains an open problem.   Recent studies have extended these approaches to fully localised patterns with dihedral symmetry via a Galerkin approximation \cite{hill2023approximate,hill2024dihedral}. These works consider the same general class of reaction--diffusion systems presented in this work, and involve performing radial spatial dynamics on a large system of differential equations. While the analysis in these studies is more complicated than the axisymmetric case, the leading order profiles for localised dihedral patterns are still made up of Bessel functions $J_{n}(r)$, and so they represent a natural extension of the results of \cite{lloyd2009localized,mccalla2013spots}. In contrast, we have seen that changing from two to three dimensions results in different leading order profiles, and so our present study represents a more conceptual problem to overcome.

\begin{figure}[t!]
    \centering
    \includegraphics[width=\linewidth]{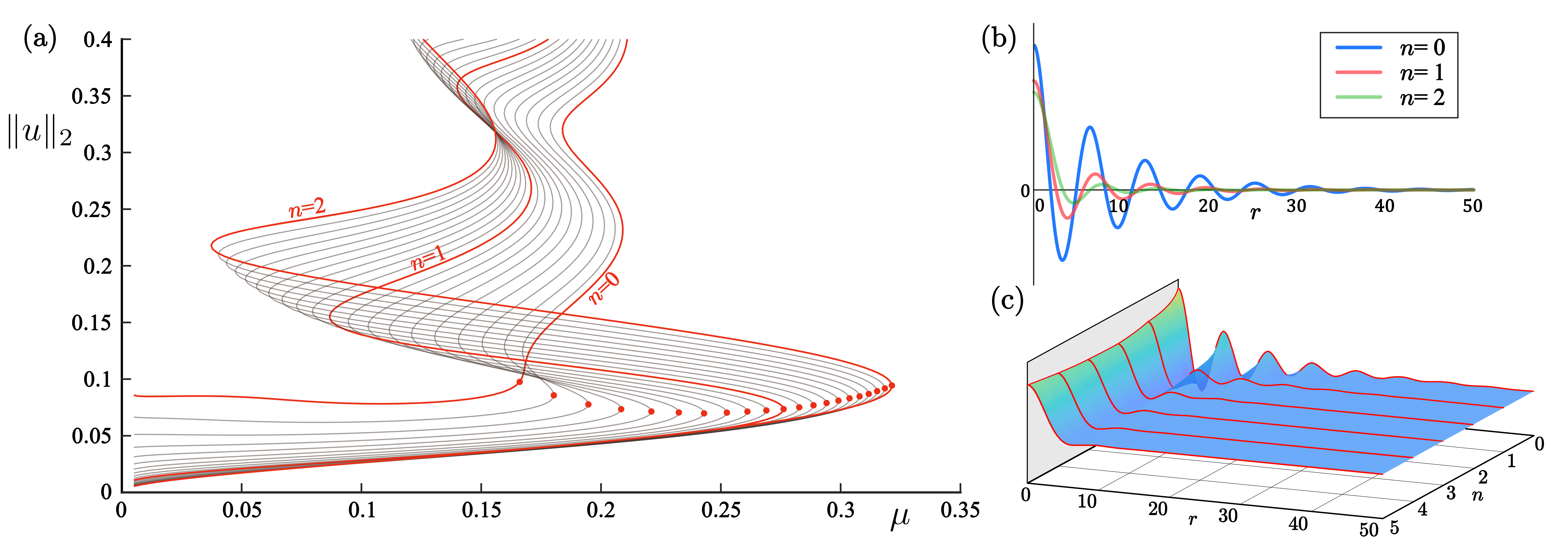}
    \caption{Spot A solutions to \eqref{eqn:SH} are computed numerically for $\nu = 1.6$ fixed and $\mu, n$ varying. (a) Bifurcation curves of Spot A patterns are plotted with respect to $\mu$ for $0\leq n\leq2$, where each curve undergoes a fold bifurcation (red spot) and before exhibiting snaking behaviour. Radial profiles of numerical solutions to \eqref{eqn:SH} at $\mu = 0.05$ are plotted for (b) $n=0,1,2$, and (c) $0<n<5$. The red lines in panel (c) represent integer values of $n$.}
    \label{fig:nSpot}
\end{figure}

Our investigation is also motivated by numerical studies of localised radial patterns. Using numerical continuation codes---such as those available at \cite{Hill2021Github}---one can solve \eqref{eqn:SH} to find localised spot A patterns for fixed values of $n$, $\nu$ and track their bifurcation curves as $\mu$ varies; see Figure~\ref{fig:nSpot}(a). We observe that the bifurcation curves in Figure~\ref{fig:nSpot}(a) continuously vary as the parameter $n$ changes; each branch grows until it reaches a fold bifurcation and then begins to exhibit snaking behaviour. As $n$ increases, the fold bifurcation becomes more pronounced, and the snaking oscillations become more distorted.

We emphasise here that the parameter $n$ does not just affect the radial profile of each localised pattern, but may also affect the asymptotic scaling of the pattern as it bifurcates from $\mu=0$. Furthermore, while there already exists a prediction for the change in the asymptotic scaling of the spot B solution, the transition between leading order radial functions (i.e. $J_0(r)$ when $n=1$ and $\frac{\sin(r)}{r}$ for $n=2$) is less clear, as shown in Figure~\ref{fig:nSpot}(b-c). The aim of this present work is to extend the approach of \cite{lloyd2009localized,mccalla2013spots} to prove the existence of $(n+1)$-dimensional localised radial solutions of \eqref{R-DEqn} near a Turing instability in a systematic fashion, such that the entire approach depends continuously on the dimension parameter $n$. We present a novel extension of standard techniques from radial spatial dynamics in order to gain a deeper understanding of the emergence of localised patterns in $(n+1)$ spatial dimensions. As a result of this approach, we not only prove the existence of spot A, spot B and ring patterns in \eqref{R-DEqn} for certain values of $n>0$, but we also determine the explicit $n$-dependence of the core profile for each pattern.

The paper is organised as follows. We begin by presenting our main results in \S\,\ref{s:results}, establishing the required hypotheses for \eqref{R-DEqn} to undergo a non-degenerate Turing bifurcation and stating our results for the existence of $(n+1)$-dimensional spot and rings patterns. In \S\,\ref{s:Spatial_Dynamics} we construct the core (\S\,\ref{subsec:Core}) and far-field (\S\,\ref{subsec:Far}-\ref{subsec:Transition}) manifolds that localised radial solutions of \eqref{R-DEqn} must intersect. We then employ asymptotic matching to identify intersections between the core and far-field manifolds (\S\,\ref{subsec:Matching}), proving our main result for stationary localised radial patterns in systems of the form \eqref{R-DEqn:Gen} near a Turing bifurcation. Finally, we conclude in \S\,\ref{s:discussion} with a discussion of our results and future directions of study.

\section{Main Results}\label{s:results}

In this paper, we assume that $n\in\mathbb{R}$, with $n>0$, and seek localised solutions to \eqref{R-DEqn}. Throughout we assume $\mathbf{D}(\mu)\in\mathbb{R}^{2\times2}$ is an invertible matrix for all $\mu$ in some open interval $\mathcal{I}\subset\mathbb{R}$, and $\mathbf{f}\in C^{k}\big(\mathbb{R}^{2}\times\mathbb{R},\mathbb{R}^{2}\big)$ for $k \geq 3$. We assume that there exists some $\mathbf{u}^*\in\mathbb{R}^{2}$ with $\mathbf{f}(\mathbf{u}^*, \mu) = \mathbf{0}$ for all $\mu\in\mathcal{I}$, such that $\mathbf{u}(r)\equiv\mathbf{u}^{*}$ is a uniform equilibrium for all $\mu \in \mathcal{I}$. Then, we restrict $\mu$ to a neighbourhood of $\mu^*$ in $\mathcal{I}$ and apply $\mathbf{D}(\mu)^{-1}$ to \eqref{R-DEqn} to normalise the diffusive term $\Delta_{n}\mathbf{u}$. 

 We proceed by writing $\mathbf{D}(\mu)^{-1}\mathbf{f}(\mathbf{u},\mu)$ as a Taylor expansion about a fixed point $(\mathbf{u},\mu) = (\mathbf{u}^*,\mu^*)$, with $\mu^*\in \mathcal{I}$, so that \eqref{R-DEqn} becomes
\begin{equation}\label{eqn:R-D}
    \mathbf{0} = \Delta_{n}\mathbf{u} - \mathbf{M}_{1}\mathbf{u} - \mu \mathbf{M}_{2}\mathbf{u} - \mathbf{Q}(\mathbf{u},\mathbf{u}) - \mathbf{C}(\mathbf{u},\mathbf{u},\mathbf{u}) + \mathcal{O}([|\mu| + |\mathbf{u}|]|\mu||\mathbf{u}| + |\mathbf{u}|^4).
\end{equation}
Here we have defined $\mathbf{M}_1 := \mathbf{D}(\mu^{*})^{-1}D_{\mathbf{u}}\big[\mathbf{f}(\mathbf{u}^*,\mu^*)\big]$ and introduced  $\mathbf{M}_2 \in\mathbb{R}^{2\times 2}$, $\mathbf{Q}:\mathbb{R}^{2}\times\mathbb{R}^{2}\to\mathbb{R}^{2}$ and $\mathbf{C}:\mathbb{R}^{2}\times\mathbb{R}^{2}\times\mathbb{R}^{2}\to\mathbb{R}^{2}$, where $\mathbf{Q}$ and $\mathbf{C}$ are symmetric bilinear and trilinear maps, respectively.   We remove the remainder terms from the rest of our analysis, as they do not affect the overall results.

  The following hypothesis establishes that \eqref{eqn:R-D}, and thus \eqref{R-DEqn}, undergoes a nondegenerate Turing instability from the uniform equilibrium $\mathbf{u} = \mathbf{u}^*$ at the point $\mu=\mu^*$.

\begin{hyp}\label{R-D:hyp} 
\,
\begin{enumerate}[label=(\roman*)]
    \item \textit{(Turing Instability)} We assume that $\mathbf{M}_{1}$ satisfies the following condition:
	\begin{equation*}
    		\det\,\big(\mathbf{M}_{1} + k_{c}^2\mathbbm{1}_2\big) = 0 
	\end{equation*}
	for some fixed $k_{c}\in\mathbb{R}^{+}$,   where $\mathbbm{1}_{2}$ is the two-dimensional identity matrix. Furthermore, the eigenvalue $\lambda=-k_{c}^{2}$ of $\mathbf{M}_{1}$ is algebraically double and geometrically simple with generalised eigenvectors $\hat U_0, \hat U_1 \in \mathbb{R}^2$, defined so that
	\begin{equation}
	    \bigg(\mathbf{M}_{1} + k_{c}^{2} \mathbbm{1}_{2}\bigg)\hat U_{0} = \mathbf{0}, \qquad \bigg(\mathbf{M}_{1} + k_{c}^{2} \mathbbm{1}_{2}\bigg)\hat U_{1} = k_{c}^{2} \hat U_0, \qquad \langle \hat U_i^*, \hat U_j\rangle_2 = \delta_{i,j},\nonumber
	\end{equation}
	for each $i,j\in\{0,1\}$, where $\hat U_0^*$, $\hat U_1^*$ are the respective adjoint vectors for $\hat U_0$, $\hat U_1$.
 \item \textit{(Linear non-degeneracy condition)} We assume that $\mathbf{M}_2$ satisfies the following condition:
	\begin{equation}\label{c0}
 \begin{split}
    c_{0}:={}& \frac{1}{4}\big\langle \hat U_1^*, -\mathbf{M}_2 \hat U_0\big\rangle_2 \neq 0.
      \end{split}
	\end{equation}
 \item \textit{(Quadratic non-degeneracy condition)} We assume that $\mathbf{Q}$ satisfies the following condition:
	\begin{equation}\label{gamma}
 \begin{split}
    \gamma:=\big\langle \hat U_1^*, \mathbf{Q}_{0,0}\big\rangle_2 \neq 0,\\
      \end{split}
	\end{equation}
 where $\mathbf{Q}_{i,j}:=\mathbf{Q}(\hat{U}_{i},\hat{U}_{j})$.
 \item \textit{(Cubic non-degeneracy condition)} We assume that $\mathbf{Q}$ and $\mathbf{C}$ satisfy the following condition:
	\begin{equation}\label{c3}
 \begin{split}
      c_3 :=&-\bigg[\bigg(\frac{5}{6}\big[\big\langle \hat{U}_{0}^{*},\mathbf{Q}_{0,0}\big\rangle_{2} + \big\langle \hat{U}_{1}^{*},\mathbf{Q}_{0,1}\big\rangle_{2}\big] + \frac{19}{18}\big\langle \hat{U}_{1}^{*},\mathbf{Q}_{0,0}\big\rangle_{2} \bigg)\big\langle \hat{U}_{1}^{*},\mathbf{Q}_{0,0}\big\rangle_{2} +\frac{3}{4}\big\langle \hat{U}_{1}^{*}, \mathbf{C}_{0,0,0}\big\rangle_{2}\bigg] \neq 0,
      \end{split}
	\end{equation}
	where $\mathbf{Q}_{i,j}:=\mathbf{Q}(\hat{U}_{i},\hat{U}_{j})$ and $\mathbf{C}_{i,j,k}:=\mathbf{C}(\hat{U}_{i},\hat{U}_{j},\hat{U}_{k})$.
\end{enumerate}	
\end{hyp}
We briefly remark on the above Hypothesis.
\begin{rmk}\,
\begin{enumerate}[label=(\roman*)]
    \item Hypothesis~\ref{R-D:hyp} $(i)$ gives that the linearisation of the right-hand-side of \eqref{R-DEqn:Gen} about $(\mathbf{u},\mu) = (\mathbf{u}^*,\mu^*)$ has a zero eigenvalue, giving way to a Turing bifurcation. We fix $k_{c}=1$ without loss of generality, the wavenumber $k_c$ can be reintroduced into our solutions through a coordinate transformation.
    \item The constant $c_0$ in Hypothesis~\ref{R-D:hyp} $(ii)$ determines the direction of bifurcation for localised radial patterns; localised patterns emerge for $\mu>0$ when $c_0>0$, and for $\mu<0$ when $c_0<0$. We assume that $c_0>0$, since we can otherwise define $\tilde{\mu}:=-\mu$ and $\tilde{c}_0:=-c_0>0$ so that localised patterns emerge for $\tilde{\mu}>0$.
    \item The non-degeneracy conditions in Hypothesis~\ref{R-D:hyp} $(iii)$, $(iv)$ are sufficient, but not necessary, conditions for localised patterns to emerge. Spot A patterns require $\gamma\neq0$ but not $c_3\neq0$, ring patterns require $c_3\neq0$ but not $\gamma\neq0$, while spot B patterns require both $\gamma\neq0$ and $c_3\neq0$.
\end{enumerate}
\end{rmk}

\begin{rmk}
    We note that the quadratic-cubic Swift--Hohenberg equation \eqref{eqn:SH} is contained within the framework of \eqref{eqn:R-D}, where $\mathbf{u} = (u,(1+\Delta_{n})u)^{T}$, $\hat{U}_{0}=(1,0)^{T}$, $\hat{U}_{1}=(0,1)^{T}$, and
    \begin{equation*}
        \mathbf{Q}(\mathbf{u},\mathbf{u}) = \begin{pmatrix}
            0 \\ \nu u^2
        \end{pmatrix},\qquad \mathbf{C}(\mathbf{u},\mathbf{u},\mathbf{u}) = \begin{pmatrix}
            0 \\ -u^3
        \end{pmatrix},\qquad \mathbf{M}_{1} = \begin{pmatrix}
            -1 & 1 \\ 0 & -1
        \end{pmatrix}, \qquad \mathbf{M}_{2} = \begin{pmatrix}
            0 & 0 \\ -1 & 0
        \end{pmatrix}.
    \end{equation*}
    Then, $c_0 = \frac{1}{4}$, $\gamma = \nu$ and $c_3 = \frac{3}{4}-\frac{19\nu^2}{18}$.
\end{rmk}

Following the approach of \cite{lloyd2009localized,mccalla2013spots}, we employ tools from the radial spatial dynamics theory of Scheel \cite{scheel2003radially} to prove the existence of localised radial solutions to \eqref{eqn:R-D}. To do this, we construct two local invariant manifolds---the `core manifold' $\mathcal{W}^{cu}_{-}(\mu)$ and the `far-field manifold' $\mathcal{W}^{s}_{+}(\mu)$---over local intervals $r\in[0,r_0]$ and $r\in[r_0,\infty)$, respectively, where the point $r_0>0$ is chosen to be large. The core manifold consists of all small-amplitude solutions of \eqref{eqn:R-D} that remain bounded as $r\to0^+$, while the far-field manifold consists of all small-amplitude solutions of \eqref{eqn:R-D} that decay exponentially fast as $r\to\infty$. After constructing each manifold, we use asymptotic matching to find trajectories that lie on the intersection of both manifolds. These trajectories consist of solutions of \eqref{eqn:R-D} that are bounded for all $r\in[0,\infty)$ and decay exponentially fast as $r\to\infty$, thus defining localised radial solutions to \eqref{eqn:R-D}.

To make use of the radial structure of \eqref{R-DEqn}, we introduce the following objects that will be useful when considering $(n+1)$-dimensional radial functions.
\begin{defn}
    We define the \underline{$k$-index Bessel operator} $\mathcal{D}_{k}$ to be the nonautonomous differential operator
    \begin{equation}\label{Dk}
        \mathcal{D}_{k}: u(r) \mapsto r^{-k}\,\partial_{r}\, (r^k u(r)) = \left(\partial_r + \tfrac{k}{r}\right)u(r) 
    \end{equation}
    for any $k\in\mathbb{R}$ (see \cite{groves2024function}) and we note that the following identities
    \begin{equation}\label{Dk:ident}
        \mathcal{D}_{i}\mathcal{D}_{j} = \mathcal{D}_{j+1}\mathcal{D}_{i-1}, \qquad\qquad \mathcal{D}_{i}(r^{{j}}u(r)) = r^{j}\mathcal{D}_{i+j}u(r),\qquad\qquad \Delta_{n} = \mathcal{D}_{n}\mathcal{D}_{0},
    \end{equation}
    hold for all $i,j\in\mathbb{R}$, where $\Delta_n$ is the $(n+1)$-dimensional Laplace operator.
    
    Additionally, we define the \underline{$(n+1)$-dimensional Bessel functions} of the first and second kind to be
\begin{equation}\label{nBessel}
    J^{(n)}_\ell(r) := 2^{(\frac{n-1}{2})}\Gamma\left(\frac{n+1}{2}\right)r^{-(\frac{n-1}{2})}J_{\ell+\frac{n-1}{2}}(r),\qquad Y^{(n)}_\ell(r) := 2^{(\frac{n-1}{2})}\Gamma\left(\frac{n+1}{2}\right)r^{-(\frac{n-1}{2})}Y_{\ell+\frac{n-1}{2}}(r),
\end{equation}
    respectively, where $J_{\nu}(r)$, $Y_{\nu}(r)$ are the respective $\nu$-th order Bessel functions of the first and second kind.
\end{defn}
The Bessel operators $\mathcal{D}_{k}$, with $k\in\mathbb{Z}$, were recently introduced in \cite{groves2024function} as the natural differential operators for planar functions of the form $\mathrm{e}^{\mathrm{i}k\theta}f_{k}(r)$. While the theory presented in \cite{groves2024function} is restricted to functions of two-dimensional polar (or three-dimensional cylindrical) coordinates, the Bessel operators are also helpful in our study of radial functions in higher spatial dimensions.

The $(n+1)$-dimensional Bessel functions (also known as \textit{hyperspherical Bessel functions}) $J_{\ell}^{(n)}(r),Y_{\ell}^{(n)}(r)$ were previously introduced in \cite{Wen1985Hyperspherical}, and have recently been employed in the study of pattern formation in nonlocal transport models for cell interaction \cite{Jewell2023patterning}; see also \S\,10 in \cite{avery2006generalized} for more information. We briefly note some important properties of these functions.

\begin{rmk}
The $(n+1)$-dimensional Bessel functions $Z^{(n)}_{\ell}(r) \in \{J^{(n)}_\ell(r), Y^{(n)}_\ell(r)\}$ are defined so that
\begin{enumerate}[label=(\roman*)]
    \item $J^{(n)}_{0}(0) = 1$ for all $n>0$,
\item $\mathcal{D}_{n-1+\ell}\, Z_{\ell}^{(n)}(r) = Z_{\ell-1}^{(n)}$,  $\mathcal{D}_{-\ell}\, Z_{\ell}^{(n)}(r) = -Z_{\ell+1}^{(n)}$ for all $n,\ell\in\mathbb{R}$,
\item $J_{\ell}^{(0)}(r) = \cos(r)$, $Y_{\ell}^{(0)}(r) = \sin(r)$, $J_{\ell}^{(1)}(r) = J_{\ell}(r)$, $Y_{\ell}^{(1)}(r) = Y_{\ell}(r)$, and $J_{\ell}^{(2)}(r) = j_{\ell}(r)$, $Y_{\ell}^{(2)}(r) = y_{\ell}(r)$, where $j_{\ell}(r), y_{\ell}(r)$ are $\ell$-th order spherical Bessel functions of the first and second kind, respectively.
\end{enumerate}
Furthermore, the $(n+1)$-dimensional Bessel functions satisfy the following generalised form of Bessel's equation
\begin{equation*}
    (\mathcal{D}_{n+\ell}\mathcal{D}_{-\ell}+1)r^{\alpha}Z_{\ell+\alpha}^{(n)}(r) = 2\alpha \, r^{\alpha-1}\,Z_{\ell+\alpha-1}^{(n)}(r),
\end{equation*}
for any $n,\ell,\alpha\in\mathbb{R}$.
\end{rmk}
The functions $\{J_{0}^{(n)}(r), rJ_{1}^{(n)}(r), Y_{0}^{(n)}(r), rY_{1}^{(n)}(r)\}$ thus form a linearly independent set of solutions to the equation $(\Delta_{n} + 1)^2 u = 0$, which corresponds to the linearisation of \eqref{eqn:R-D} at $\mu=0$. In Lemma~\ref{Lemma:Core} we use a variation-of-constants formula with this set of linear solutions in order to construct core manifold $\mathcal{W}^{cu}_{-}(\mu)$ on a bounded interval $r\in[0,r_0]$. We emphasise that our use of this set of functions is not an assumption, but rather an enforced choice based on the structure of the linear operator $(\Delta_{n} - \mathbf{M}_{1})\mathbf{u}=\mathbf{0}$.

In order to construct the far-field manifold $\mathcal{W}^{s}_{+}(\mu)$, we introduce complex coordinates $A,B$ so that the system \eqref{eqn:R-D} is transformed into the radial normal form for an $(n+1)$-dimensional Hamiltonian--Hopf bifurcation. Upon introducing the far-field radial coordinate $s:=\mu^{\frac{1}{2}}r$, exponentially decaying solutions are then found to be $\mathcal{O}(\mu^{\frac{1}{2}})$-perturbations of the non-autonomous cubic Ginzburg--Landau equation
\begin{equation}\label{GL:intro}
        \left(\frac{\mathrm{d}}{\mathrm{d}s} + \frac{n}{2s}\right)^{2} A = c_0 A + c_3 |A|^2 A, \qquad A\in\mathbb{C},\qquad s\in[0,\infty),
    \end{equation}
where $c_0$ and $c_3$ are defined in \eqref{c0} and \eqref{c3}, respectively. Localised ring and spot B solutions require the existence of a nontrivial forward-bounded localised solution to \eqref{GL:intro}, and so the authors of  \cite{lloyd2009localized,mccalla2013spots} employed the following hypothesis.
\begin{hyp}\label{GL:hyp}
    Fix $n>0$. The equation
    \begin{equation}\label{GL:real;normal}
        \left(\frac{\mathrm{d}}{\mathrm{d}s} + \frac{n}{2s}\right)^{2} A = A - A^3, \qquad A\in\mathbb{R},\qquad s\in[0,\infty)
    \end{equation}
    has a forward-bounded nontrivial localised solution $A(s) = q(s)$. In addition, the linearisation of \eqref{GL:real;normal} about $q(s)$ possesses no nontrivial solutions that are forward-bounded on $[0,\infty)$.
\end{hyp}
It is straightforward to prove this hypothesis when $n=2$ as the differential operator becomes the three-dimensional radial Laplacian and one can employ theory from the vast literature on elliptic PDEs (see Lemma 4.3 in \cite{mccalla2013spots} and the references therein). However, it is not clear that this result holds for any other values of $n>0$, this hypothesis was eventually proven in the case when $n=1$ by van den Berg et al.\ in \cite{vandenberg2015Rigorous} via a computer-assisted proof. Based on the expected asymptotic scaling $|\mathbf{u}|\sim\mu^{\frac{4-n}{8}}$ for spot B solutions in \cite{mccalla2013spots}, as well as a similar prediction of $|\mathbf{u}|\sim\mu^{\frac{4-n}{4}}$ for ring solutions, we only consider solutions to \eqref{GL:intro} for $0<n<4$. We will observe that this is required in order to ensure certain terms are asymptotically small in our subsequent analysis.

In this work, we present a simplified version of Hypothesis~\ref{GL:hyp} which can be proven for any choice of $0<n<3$. Our proof relies on the following proposition which can be proven using standard elliptic PDE theory.
\begin{prop}\label{prop:GL-Laplace}
    The equation
    \begin{equation}\label{GL:real;Laplace}
        \Delta u = u - |\mathbf{x}|^{2-n} u^3, \qquad u\in\mathbb{R},\qquad \mathbf{x}\in\mathbb{R}^{3}
    \end{equation}
    has a positive radial ground state solution $u(\mathbf{x}) = Q(|\mathbf{x}|)$ for any $0<n<3$. In addition, the linearisation of \eqref{GL:real;Laplace} about $Q(|\mathbf{x}|)$ 
    \begin{equation}\label{GL:real;Laplace-lin}
        \mathcal{L}\,v := \left(-\Delta + 1 - 3|\mathbf{x}|^{2-n} Q(|\mathbf{x}|)^2 \right) v = 0, \qquad v\in\mathbb{R},\qquad \mathbf{x}\in\mathbb{R}^{3}
    \end{equation}
    possesses no nontrivial radial solutions that are bounded uniformly on $[0,\infty)$.
\end{prop}
\begin{proof}
      Existence of a radial positive ground state follows from standard variational methods, see for example \cite{Lions1984concentration-compactness-1,Lions1984concentration-compactness-2,Stuart1980variational}. In particular, existence was proven for $2<n<3$ in \cite{Genoud_2008_Schrodinger}, and follows directly from \cite{Ding1986Existence} for $0<n\leq2$. The nonexistence of bounded solutions to \eqref{GL:real;Laplace-lin} follows from the proof of Lemma 2.1 in \cite{Chang2008Spectra}, where we use the fact that the two quantities $Q(r)$, $Q_1(r) := r \mathcal{D}_{\frac{4-n}{2}}Q(r)$ satisfy $\mathcal{L}Q(|\mathbf{x}|) = -2 |\mathbf{x}|^{2-n}Q(|\mathbf{x}|)^3$ and $\mathcal{L}Q_{1}(|\mathbf{x}|) = -2Q(|\mathbf{x}|)$.
\end{proof}

In Lemma~\ref{lem:GL} we show that Proposition~\ref{prop:GL-Laplace} implies Hypothesis~\ref{GL:hyp} and thus guarantees the existence of localised ring and spot B solutions to \eqref{eqn:R-D} for $0<n<3$.   We also show the existence of localised ring solutions for $3\leq n<4$, subject to the following hypothesis.
\begin{hyp}\label{hyp:GL-Laplace}
    Proposition~\ref{prop:GL-Laplace} also holds for any value $3\leq n<4$.
\end{hyp}
The standard techniques used to prove the existence of $H^1(\mathbb{R}^{3})$ (and consequently $C^2(\mathbb{R}^{3})$) radial solutions in Proposition~\ref{prop:GL-Laplace} break down when $n>3$. In fact, one can prove the nonexistence of $H^1(\mathbb{R}^{3})$ solutions to \eqref{GL:real;Laplace} in this regime. It may still be possible to obtain solutions in a different function space that guarantees the existence of localised ring solutions for $3\leq n <4$, however we leave this question to future study and note the following. For the remainder of this work, the restriction $0<n<4$ for the existence of ring and spot B solutions is to be understood as either $0<n<3$ or $3\leq n<4$ subject to Hypothesis~\ref{hyp:GL-Laplace}.

We present our key result in the following theorem.

\begin{thm}\label{thm:SmallPatch} 
	Fix $n>0$ and assume Hypotheses~\ref{R-D:hyp} (i-ii) hold. Then, there exist constants $\mu_0,r_0,r_1 > 0$ such that the system \eqref{eqn:R-D} has radially localised solutions for each $\mu\in(0,\mu_{0})$. Each respective solution decays at a rate $\mathcal{O}(r^{-\frac{n}{2}}\mathrm{e}^{-\sqrt{c_0\mu}\,r})$ as $r\to\infty$ and satisfies the following equations uniformly on bounded intervals $[0,r_0]$ as $\mu\to0^+$.
 \begin{enumerate}[label=(\roman*)]
     \item (Spot A) If Hypothesis~\ref{R-D:hyp} (iii) holds, then we obtain
     \begin{equation}\label{RadialProfile:spotA}
\begin{split}
    \mathbf{u}_{A}(r) ={}& (c_0 \mu)^{\frac{1}{2}}\left(\frac{\sqrt{\pi} }{\nu_{n}\gamma}\right)\frac{1}{2^{\frac{n}{2}}\Gamma(\frac{n+1}{2})} J_{0}^{(n)}(r)\hat{U}_0 + \mathcal{O}(\mu),\\
\end{split}
\end{equation} 
\item (Rings) If $n<4$ and Hypothesis~\ref{R-D:hyp} (iv) holds with $c_3<0$, then we obtain
\begin{equation}\label{RadialProfile:Rings}
\begin{split}
    \mathbf{u}_{R}(r) ={}& \pm (c_0 \mu)^{\frac{4-n}{4}} \left(\frac{2\sqrt{\pi}q_n}{\sqrt{|c_3|}}\right) \frac{1}{2^{\frac{n}{2}}\Gamma(\frac{n+1}{2})} \left[ r J_{1}^{(n)}(r)\hat{U}_0 + 2J_{0}^{(n)}(r)\hat{U}_1 \right] + \mathcal{O}(\mu^{\frac{6-n}{4}}+\mu^{\frac{4-n}{2}}),\\
\end{split}
\end{equation} 
\item (Spot B) If $n<4$ and Hypotheses~\ref{R-D:hyp} (iii-iv) hold with $c_3<0$, then we obtain
\begin{equation}\label{RadialProfile:spotB}
\begin{split}
    \mathbf{u}_{B}(r) ={}& -\mathrm{sgn}(\gamma) (c_0\mu)^{\frac{4-n}{8}} \left(\frac{\pi q_n}{\nu_n |\gamma|\sqrt{|c_3|}}\right)^{\frac{1}{2}} \frac{1}{2^{\frac{n-1}{2}}\Gamma(\frac{n+1}{2})} J_{0}^{(n)}(r)\hat{U}_0 + \mathcal{O}(\mu^{\frac{8-n}{8}}+\mu^{\frac{4-n}{4}}),\\
\end{split}
\end{equation} 
 \end{enumerate}
Here, $\gamma,c_0,c_3$ are defined in Hypothesis~\ref{R-D:hyp}, $q_n := \displaystyle\lim_{s\to0} Q(s)$, where $Q(|\mathbf{x}|)$ is the positive radial ground state solution of \eqref{GL:real;Laplace}, 
\begin{equation}\label{nu_n}
    \nu_{n}:= \left(\frac{3}{8}\right)^{\frac{n}{2}}\frac{\pi}{3\Gamma(\frac{n}{2})},
\end{equation}
and $J^{(n)}_{\ell}(r)$ is the $(n+1)$-dimensional $\ell^\mathrm{th}$ order Bessel function of the first kind defined in \eqref{nBessel}. 

Furthermore, for $c_3>0$, spot A solutions given by \eqref{RadialProfile:spotA} undergo a fold bifurcation along the curve
        \begin{equation}\label{Fold:curve}
            \gamma 
            = \pm (c_0 \mu)^{\frac{1}{4}} \left(\frac{(\mu^{-\frac{1}{2}}r_1)^{n-1} - r_0^{n-1}}{(n-1)(\mu^{-\frac{1}{2}} r_1)^{n-1}\,r_0^{n-1}}\right)^{\frac{1}{2}}\,\left(\frac{|c_3|^{\frac{1}{2}}}{\nu_{n}}\right)(1 + \mathrm{o}(1))
        \end{equation}
        as $\mu\to0^+$.
\end{thm}

    For planar patterns $\nu_n = \frac{1}{2}\sqrt{\frac{\pi}{6}}$ and for spherical patterns $\nu_n = \frac{\pi}{8}$, which are both in agreement with \cite{lloyd2009localized,mccalla2013spots}. The implicit curve \eqref{Fold:curve} implies that, for a fixed $0<\gamma\ll1$, the fold point $\mu=\mu_*$ with $0<\mu_*\ll1$ must increase as $n$ increases. This is in agreement with Figure~\ref{fig:nSpot}, although we note that the parameter values in Figure~\ref{fig:nSpot} are far from the values where our analysis holds. While we are unable to derive a closed form expression for $q_n$ in the same way as for $\nu_n$, we can numerically compute its value for different choices of $n$ by solving \eqref{GL:real;normal} via finite difference techniques. We present plots of both $\nu_n$ and $q_n$ in Figure~\ref{fig:n_constants}.

\begin{figure}[t!]
    \centering
    \includegraphics[width=0.8\linewidth]{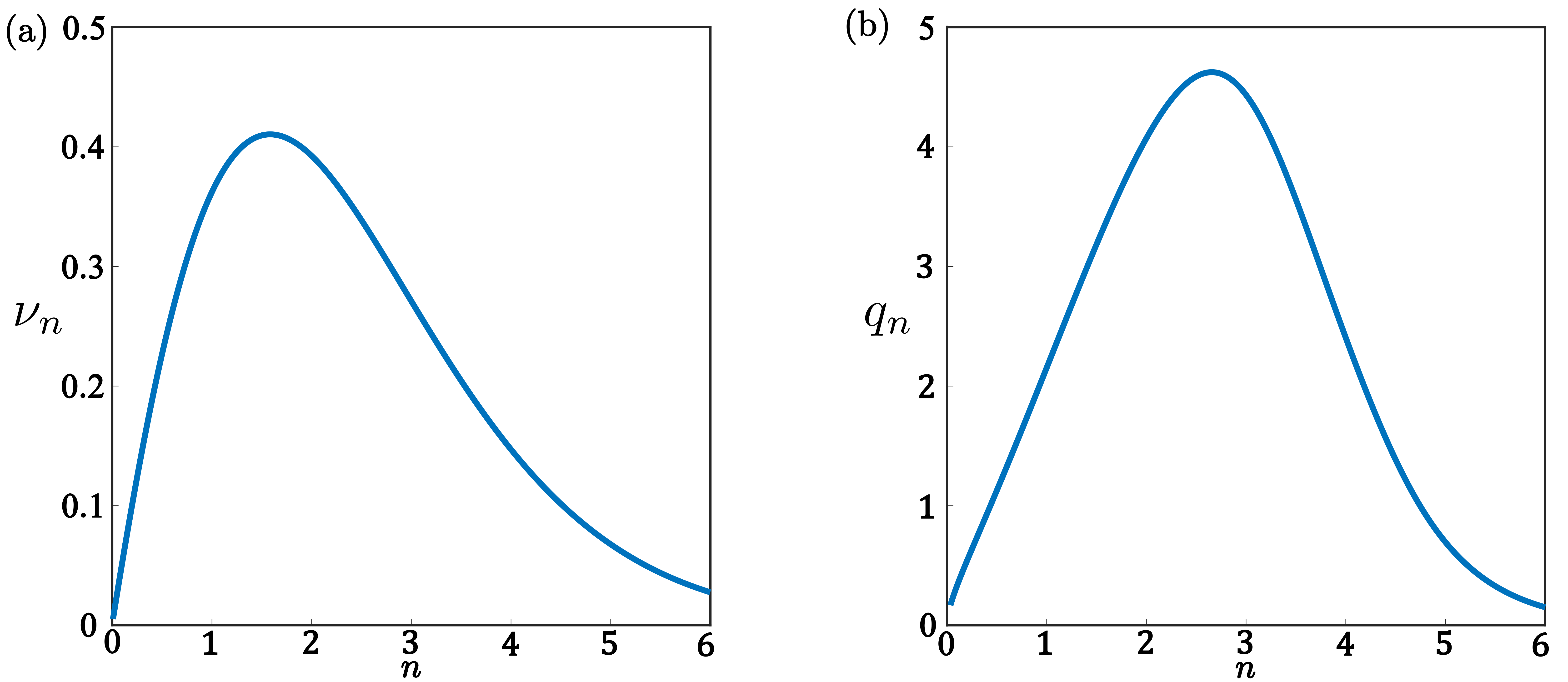}
    \caption{A plot of the values of (a) $\nu_n$ and (b) $q_n$ for $0<n<6$.}
    \label{fig:n_constants}
\end{figure}
%
%
%
%
%
%
%
%
%
%
\section{Proof of Theorem~\ref{thm:SmallPatch}}\label{s:Spatial_Dynamics}
We begin by expressing \eqref{eqn:R-D} as the following first-order system,
\begin{equation}\label{R-D:U;vec}
    \frac{\textnormal{d}}{\textnormal{d} r}\mathbf{U} = \mathcal{A}(r)\mathbf{U} + \mathbf{F}(\mathbf{U},\mu), 
\end{equation}
where $\mathbf{U}:=(\mathbf{u},\mathcal{D}_{0}\mathbf{u})^{T}$, with
\begin{align}
    \mathcal{A}(r) &= \begin{pmatrix}
    \mathbb{O}_{2} & \mathbbm{1}_{2} \\ \mathbf{M}_{1} & -\frac{n}{r}\mathbbm{1}_{2}
    \end{pmatrix},\quad  \mathbf{F}(\mathbf{U};\mu) = \begin{pmatrix}
    \mathbf{0} \\ \displaystyle \mu \mathbf{M}_{2}\mathbf{u} + \mathbf{Q}(\mathbf{u},\mathbf{u}) + \mathbf{C}(\mathbf{u},\mathbf{u},\mathbf{u})
    \end{pmatrix},\nonumber
\end{align}
where we note that $\mathbb{O}_2$ and $\mathbbm{1}_2$ are the zero and identity matrices in $\mathbb{R}^{2\times2}$, respectively.

\subsection{The Core Manifold}\label{subsec:Core}

  We begin by characterising the core manifold $\mathcal{W}_{-}^{cu}(\mu)$, which forms a local invariant manifold containing all small-amplitude bounded solutions to \eqref{eqn:R-D} on the interval $r\in[0,r_{0}]$, for some large $r_{0}>0$. At the bifurcation point $\mu=0$, the linearisation of \eqref{R-D:U;vec}
has solutions of the form
\begin{equation}\label{Ntup:SH;lin,soln}
	\mathbf{V}(r) = \sum_{j=1}^{4}  d_{j} \mathbf{V}_{j}(r), 
\end{equation}
with 
\begin{align}
   \begin{split} \mathbf{V}_{1}(r) &= \frac{\sqrt{\pi}}{2^{\frac{n}{2}}\Gamma(\frac{n+1}{2})}\begin{pmatrix}
    J^{(n)}_{0}(r)\hat{U}_{0} \\ -J^{(n)}_{1}(r)\hat{U}_{0}
    \end{pmatrix}, \\ \mathbf{V}_{2}(r) &= \frac{\sqrt{\pi}}{2^{\frac{n}{2}}\Gamma(\frac{n+1}{2})}\begin{pmatrix}
    r J^{(n)}_{1}(r) \hat{U}_{0} + 2 J^{(n)}_{0}(r)\hat{U}_{1} \\ [r J^{(n)}_{0}(r)-(n-1)J_{1}^{(n)}(r)]\hat{U}_{0} - 2  J^{(n)}_{1}(r)\hat{U}_{1}
    \end{pmatrix}, \\
    \mathbf{V}_{3}(r) &= \frac{\sqrt{\pi}}{2^{\frac{n}{2}}\Gamma(\frac{n+1}{2})}\begin{pmatrix}
    Y^{(n)}_{0}(r)\hat{U}_{0} \\ -Y^{(n)}_{1}(r)\hat{U}_{0}
    \end{pmatrix}, \\ \mathbf{V}_{4}(r) &= \frac{\sqrt{\pi}}{2^{\frac{n}{2}}\Gamma(\frac{n+1}{2})}\begin{pmatrix}
    r Y^{(n)}_{1}(r) \hat{U}_{0} + 2  Y^{(n)}_{0}(r)\hat{U}_{1} \\ [r Y^{(n)}_{0}(r)-(n-1)Y_{1}^{(n)}(r)]\hat{U}_{0} - 2  Y^{(n)}_{1}(r)\hat{U}_{1}
    \end{pmatrix}.\end{split}
\end{align}
   These solutions follow by first noting that the system $\frac{\textnormal{d}}{\textnormal{d} r}\mathbf{V} = \mathcal{A}(r)\mathbf{V}$ is equivalent to 
\begin{equation}
    \left(\mathcal{D}_{n}\mathcal{D}_{0} - \mathbf{M}_{1}\right)\mathbf{u} = \mathbf{0}
\end{equation}
with $\mathbf{V} =(\mathbf{u},\mathcal{D}_{0}\mathbf{u})^{T}$. By taking the eigenbasis decomposition $\mathbf{u} = u_1\hat{U}_0 + u_2\hat{U}_1$, we then obtain
\begin{equation}
    \left(\mathcal{D}_{n}\mathcal{D}_{0} + 1\right)u_1 = u_2,\qquad \qquad \left(\mathcal{D}_{n}\mathcal{D}_{0} + 1\right)u_2 = 0,
\end{equation}
which has solutions $u_2 = 2c_2 J_0^{(n)}(r) + 2c_4 Y_n^{(n)}(r)$ and $u_1 = c_1 J_0^{(n)}(r) + c_2 r J_1^{(n)}(r) + c_3 Y_n^{(n)}(r) + c_4 r Y_1^{(n)}(r)$. 

We now present a standard result regarding the existence and parametrisation of $\mathcal{W}^{cu}_{-}(\mu)$ on the interval $[0,r_0]$, for $0<\mu\ll1$. We use the notation $\mathcal{O}_{r_{0}}(\cdot)$ to denote the standard Landau symbol $\mathcal{O}(\cdot)$ with bounding constants that depend on the value of $r_{0}$. Additionally, we define $\mathcal{P}^{cu}_{-}(r_0)$ to be a projection onto the subspace spanned by bounded (as $r\to0$) solutions $\{\mathbf{V}_{1}(r_0), \mathbf{V}_{2}(r_0)\}$, evaluated at the point $r=r_0$.

\begin{table} 
\centering
\begin{tabular}{|c|c|c|}
\hline
     Function & as $r\to0$ & as $r\to\infty$ \\
     \hline
     & & \\
     $J^{(n)}_\ell(r)$ & $\mathcal{O}\left(r^{\ell}\right)$ & $\frac{2^{\frac{n}{2}}\Gamma\left(\frac{n+1}{2}\right)}{\sqrt{\pi}} r^{-\frac{n}{2}}\cos\left(r - \frac{n\pi}{4} - \frac{\ell\pi}{2}\right) + \mathcal{O}\left(r^{-(\frac{n+2}{2})}\right)$\\
     & & \\
     $Y^{(n)}_{\ell}(r)$ & $\mathcal{O}\left(r^{\ell}\left(\cot(\nu\pi) - \frac{2^{2\nu}\Gamma(\nu)\Gamma(\nu+1)}{r^{2\nu} \pi}\right)_{\nu=\ell + \frac{n-1}{2}}\right)$ & $\frac{2^{\frac{n}{2}}\Gamma\left(\frac{n+1}{2}\right)}{\sqrt{\pi}} r^{-\frac{n}{2}}\sin\left(r - \frac{n\pi}{4} - \frac{\ell\pi}{2}\right) + \mathcal{O}\left(r^{-(\frac{n+2}{2})}\right)$\\
     & & \\
     \hline
\end{tabular}
\caption{Orders and expansions for $(n+1)$-dimensional $\ell$'th-order Bessel functions of the first and second kind as $r\to0$ and $r\to\infty$, respectively, where $n\in\mathbb{R}$, $n\geq0$; see \cite[(9.1.10),(9.1.11), and \S9.2]{abramowitz1972handbook}.}
\label{table:Bessel}
\end{table} 

\begin{lem}\label{Lemma:Core} 
 For each fixed $r_{0}>0$, there are constants $\delta_{1},\delta_{2}>0$ such that the set $\mathcal{W}^{cu}_{-}(\mu)$ of solutions $\mathbf{U}(r)$ of \eqref{R-D:U;vec} for which $\sup_{0\leq r\leq r_{0}}\|\mathbf{U}(r)\|<\delta_{1}$ is, for $|\mu|<\delta_{1}$, a smooth $2$ dimensional manifold. Furthermore, each $\mathbf{U}(r_{0})\in\mathcal{W}^{cu}_{-}(\mu)$ with $|\mathcal{P}^{cu}_{-}(r_{0})\mathbf{U}(r_{0})|<\delta_{2}$ can be written uniquely as 
\begin{equation}\label{U:Core;Ntup}
	\begin{split}
 		\mathbf{U}(r_{0}) &= d_{1}\mathbf{V}_{1}(r_{0}) + d_{2}\mathbf{V}_{2}(r_{0}) + \mathbf{V}_{3}(r_{0})\mathcal{O}_{r_0}(|\mu||\mathbf{d}| + |\mathbf{d}|^{2})\\
 &\qquad \qquad  + \mathbf{V}_{4}(r_{0})\left[\left(\nu_n + \mathcal{O}\left(r_0^{-\frac{n}{2}}\right)\right)\gamma\, d_1^2 + \mathcal{O}_{r_0}(|\mu||\mathbf{d}| + |d_{2}|^{2} + |d_{1}|^{3})\right],
	\end{split}
\end{equation}
where $\gamma$ is defined in \eqref{gamma}, $\nu_n$ is defined in \eqref{nu_n} for all $n>0$, $\mathbf{d}:=\left(d_{1},d_{2}\right)\in\mathbb{R}^{2}$ with $|\mathbf{d}|<\delta_{2}$, and the right-hand side of \eqref{U:Core;Ntup} depends smoothly on $(\mathbf{d},\mu)$.
\end{lem}

\begin{proof}
This statement is proven in a similar way to several previously proven lemmata; see, for example, \cite[Lemma~1]{lloyd2009localized} or \cite[Lemma~4.1]{hill2023approximate}.   We first note that the linear adjoint problem $\frac{\textnormal{d}}{\textnormal{d}r}\mathbf{W} = -\mathcal{A}^{T}(r)\mathbf{W}$ is equivalent to $(\mathcal{D}_{n}\mathcal{D}_{0} - \mathbf{M}_{1}^{T})\mathbf{w}=\mathbf{0}$, where $\mathbf{W} = (-r^n \mathcal{D}_{0}\mathbf{w},r^n \mathbf{w})^{T}$, and thus has independent solutions of the form
\begin{align}
\begin{split}
    \mathbf{W}_{1} &= \frac{\sqrt{\frac{\pi}{2}}}{2^{(\frac{n+1}{2})}\Gamma(\frac{n+1}{2})}\begin{pmatrix}
    r^n\left[ r Y^{(n)}_{0}(r)\hat{U}_{1}^{*} - 2 Y^{(n)}_{1}(r)\hat{U}_{0}^{*} \right] \\ -r^n\left[[r Y^{(n)}_{1 }(r) - (n-1) Y_{n}^{(n)}(r)]\hat{U}_{1}^{*} + 2 Y^{(n)}_{0}(r)\hat{U}_{0}^{*}\right] 
    \end{pmatrix}, \\ \mathbf{W}_{2} &= \frac{\sqrt{\frac{\pi}{2}}}{2^{(\frac{n+1}{2})}\Gamma(\frac{n+1}{2})}\begin{pmatrix}
    - r^n Y^{(n)}_{1}(r)\hat{U}_{1}^{*} \\ -r^n Y^{(n)}_{0}(r)\hat{U}_{1}^{*}
    \end{pmatrix}, \\
    \mathbf{W}_{3} &= \frac{\sqrt{\frac{\pi}{2}}}{2^{(\frac{n+1}{2})}\Gamma(\frac{n+1}{2})}\begin{pmatrix}
     -r^n\left[ r J^{(n)}_{0}(r) \hat{U}_{1}^{*} - 2 J^{(n)}_{1}(r)\hat{U}_{0}^{*}\right] \\ r^n\left[[r J^{(n)}_{1}(r) - (n-1)J_{0}^{(n)}(r)] \hat{U}_{1}^{*} + 2  J^{(n)}_{0}(r)\hat{U}_{0}^{*}\right]
    \end{pmatrix}, \\ \mathbf{W}_{4} &=  \frac{\sqrt{\frac{\pi}{2}}}{2^{(\frac{n+1}{2})}\Gamma(\frac{n+1}{2})}\begin{pmatrix}
    r^n J^{(n)}_{1}(r) \hat{U}_{1}^{*}\\ r^n J^{(n)}_{0}(r)\hat{U}_{1}^{*}
    \end{pmatrix}.
    \end{split}\label{Ntup:R-D;adj,soln}
\end{align}
such that the relation
\begin{equation}
    \langle \mathbf{W}_{i}(r), \mathbf{V}_{j}(r)\rangle_{4}  = \delta_{i,j}
\end{equation}
holds for all $r>0$, $i,j \in\{1,2,3,4\}$. The orthonormal property follows in part from the identity
\begin{equation*}
    r^n[J_{1}^{(n)}(r)Y_{0}^{(n)}(r) - J_{0}^{(n)}(r)Y_{1}^{(n)}(r)] 
    = \frac{2^n \Gamma(\tfrac{n+1}{2})^2}{\pi}.
\end{equation*}
For a given $\mathbf{d}=(d_{1},d_{2})\in\mathbb{R}^{2(N+1)}$, we consider the fixed-point equation
\begin{equation}\label{var:U}
	\begin{split}
    	\mathbf{U}(r) &= \sum_{j=1}^{2}  d_{j} \mathbf{V}_{j}(r) + \sum_{j=1}^{2} \mathbf{V}_{j}(r)\int_{r_{0}}^{r} \langle \mathbf{W}_{j}(s), \mathbf{F}(\mathbf{U}[s];\mu) \rangle_{4} \,\textnormal{d} s + \sum_{j=3}^{4} \mathbf{V}_{j}(r)\int_{0}^{r}  \langle \mathbf{W}_{j}(s), \mathbf{F}(\mathbf{U}[s];\mu) \rangle_{4} \,\textnormal{d} s, 
    	\end{split}
\end{equation}
  on $C([0,r_{0}],\mathbb{R}^{4})$, whose solutions define all solutions to \eqref{R-D:U;vec} that remain bounded on $[0,r_0]$. After proving that \eqref{var:U} defines a contraction for sufficiently small $\mathbf{d}$ and $\mu$ and evaluating at $r=r_{0}$, we arrive at
\begin{equation}\label{var:U;r0}
    \mathbf{U}(r_{0}) = \sum_{j=1}^{2}  d_{j} \mathbf{V}_{j}(r_{0}) + \sum_{j=3}^{4} \mathbf{V}_{j}(r_{0})\int_{0}^{r_{0}}  \langle \mathbf{W}_{j}(s), \mathbf{F}(\mathbf{U}[s];\mu) \rangle_{4} \,\textnormal{d} s . 
\end{equation}
Then, we introduce
\begin{equation}
    c_{j}\left(d_1, d_2;\mu\right) := \int_{0}^{r_{0}} \langle \mathbf{W}_{j}(s), \mathbf{F}(\mathbf{U}[s];\mu) \rangle_{4} \,\textnormal{d} s,\label{cn:defn}
\end{equation}
for $j=3,4$ so that we can write our small-amplitude core solution as
\begin{equation}
    	\mathbf{U}(r_{0}) = d_{1} \mathbf{V}_{1}(r_{0}) +  d_{2} \mathbf{V}_{2}(r_{0}) +  c_{3}\left(d_{1}, d_{2}; \mu\right) \mathbf{V}_{3}(r_{0}) + c_{4}\left(d_{1}, d_{2}; \mu\right) \mathbf{V}_{4}(r_{0}). 
\end{equation}
In order to arrive at \eqref{U:Core;Ntup}, we apply a Taylor expansion to \eqref{cn:defn} about $|d_{1}| = |d_{2}|=\mu=0$ and find 
\begin{equation}
	\begin{split}
    		c_{3}\left(d_{1},d_{2};\mu\right) &= \mathcal{O}_{r_{0}}\left(|\mu||\mathbf{d}| + |\mathbf{d}|^{2}\right), \\
    		c_{4}\left(d_{1},d_{2};\mu\right) &= \left[\nu_n + \mathcal{O}\left(r_0^{-\frac{n}{2}}\right)\right]\gamma\, d_1^2 + \mathcal{O}_{r_{0}}\left(|\mu||\mathbf{d}| + |d_{2}|^{2} + |d_{1}|^{3}\right), \\
    	\end{split}
\end{equation}
where $\gamma=\langle \hat{U}_1^*, \mathbf{Q}(\hat{U}_0,\hat{U}_0) \rangle_{2}$ and
\begin{equation}
    \nu_{n} := \frac{1}{2}\left(\frac{\sqrt{\frac{\pi}{2}}}{2^{(\frac{n-1}{2})}\Gamma(\frac{n+1}{2})}\right)^3 \int_{0}^{\infty} s^n (J_0^{(n)}(s))^3 \,\textnormal{d} s.
\end{equation}
Here, we have used the fact that
\begin{equation*}
    \int_{r_0}^{\infty} s^n (J_0^{(n)}(s))^3 \,\textnormal{d} s = \mathcal{O}\left(r_0^{-\frac{n}{2}}\right)
\end{equation*}
for $n>0$ and $r_0\gg1$ in order to determine the order of the remainder terms. We compute the explicit value of $\nu_n$, obtaining
 \begin{equation}
\begin{split}
    \nu_{n}
    ={}& \frac{\pi\sqrt{\pi}}{4\sqrt{2}}\int_{0}^{\infty} s^{1-(\frac{n-1}{2})}(J_{\frac{n-1}{2}}(s))^3 \,\textnormal{d} s = \left(\frac{3}{8}\right)^{\frac{n}{2}}\frac{\pi}{3\Gamma(\frac{n}{2})}.\\
\end{split}
\end{equation}
This completes the proof.
\end{proof} 

  For sufficiently large values of $r_{0}$, Table~\ref{table:Bessel} allows us to write down the following leading-order expansion for \eqref{U:Core;Ntup}
\begin{equation}\label{Core:un}
	\begin{split}
    \mathbf{u}(r_{0}) &= r_0^{-\frac{n}{2}}\left[d_2 r_0\left(1 + \mathcal{O}(r_0^{-1})\right)\sin\left(y_{n}\right) + d_1 \left(1 + \mathcal{O}(r_0^{-1})\right)\cos\left(y_{n}\right) + \mathcal{O}_{r_0}(|\mu||\mathbf{d}| + |\mathbf{d}|^{2})\right]\hat{U}_0\\
   &\qquad + 2r_0^{-\frac{n}{2}}\Big[\gamma\, d_1^2\left(\nu_n + \mathcal{O}(r_0^{-\frac{n}{2}}+r_0^{-1})\right) \sin\left(y_{n}\right) + d_2 \left(1 + \mathcal{O}(r_0^{-1})\right)\cos\left(y_{n}\right)\\
   &\hspace{8.2cm}  + \mathcal{O}_{r_0}(|\mu||\mathbf{d}| + |d_2|^{2} + |d_1|^3)\Big]\hat{U}_1, \\
 \mathbf{v}(r_{0}) &= r_0^{-\frac{n}{2}}\left[ d_2 r_0 \left(1 + \mathcal{O}(r_0^{-1})\right)\cos\left(y_{n}\right) -d_1\left(1 + \mathcal{O}(r_0^{-1})\right)\sin\left(y_{n}\right) + \mathcal{O}_{r_0}(|\mu||\mathbf{d}| + |\mathbf{d}|^{2})\right]\hat{U}_0\\
   &\qquad + 2r_0^{-\frac{n}{2}}\Big[\gamma\, d_1^2\left(\nu_n + \mathcal{O}(r_0^{-\frac{n}{2}}+r_0^{-1})\right) \cos\left(y_{n}\right) - d_2 \left(1 + \mathcal{O}(r_0^{-1})\right)\sin\left(y_{n}\right) \\
   & \hspace{8.2cm} + \mathcal{O}_{r_0}(|\mu||\mathbf{d}| + |d_2|^{2} + |d_1|^3)\Big]\hat{U}_1. \\
	\end{split}
\end{equation}
  Here we have defined $y_{n}:=r_{0} - \frac{n\pi}{4}$, and the remainder terms $\mathcal{O}_{r_{0}}(\cdot)$ capture the higher order terms for small values of $|\mathbf{d}|$ and $|\mu|$.

\subsection{The Far-Field Manifold}\label{subsec:Far}

  We seek to find all small-amplitude exponentially-decaying solutions to \eqref{eqn:R-D}, thereby constructing the far-field manifold $\mathcal{W}_{+}^{s}(\mu)$. We introduce the variable $\sigma(r)\geq0$ to replace any $\frac{1}{r}$ terms in \eqref{R-D:U;vec}, so that we obtain the autonomous extended system
\begin{equation}\label{R-D:Farf}
	\begin{split}
    		\frac{\textnormal{d}}{\textnormal{d} r}\mathbf{U} &= \mathcal{A}(\infty)\mathbf{U} + \widetilde{\mathbf{F}}(\mathbf{U}; \mu, \sigma),\\
    		\frac{\textnormal{d}}{\textnormal{d} r}\sigma &= -\sigma^{2},
	\end{split}
\end{equation}
  with the property that $\sigma(r)= r^{-1}$ is an invariant manifold of \eqref{R-D:Farf}. This invariant manifold clearly recovers the non-autonomous system \eqref{R-D:U;vec}, and so any solutions to \eqref{R-D:Farf} where $|\mathbf{U}|\to0$ exponentially fast as $r\to\infty$ and $\sigma(r_0) = r_0^{-1}$ is an exponentially decaying solution of \eqref{R-D:U;vec}.

  We first transform \eqref{R-D:Farf} into a radial normal form so that we can identify exponentially decaying solutions. We introduce complex amplitudes $\widetilde{A}, \widetilde{B}$, where
\begin{equation}
    \begin{aligned}
    \mathbf{u} &= (\widetilde{A} + \overline{\widetilde{A}})\hat{U}_{0} + 2\textnormal{i}(\widetilde{B} - \overline{\widetilde{B}})\hat{U}_{1} , &\qquad \mathbf{v} &= \big[\textnormal{i}(\widetilde{A} - \overline{\widetilde{A}}) + (\widetilde{B} + \overline{\widetilde{B}})\big]\hat{U}_{0} - 2(\widetilde{B} + \overline{\widetilde{B}})\hat{U}_{1},\\
    \widetilde{A} &= \frac{1}{2}\big\langle \hat{U}_{0}^{*}, \mathbf{u} - \textnormal{i}\mathbf{v}\big\rangle_{2} - \frac{\textnormal{i}}{4}\big\langle \hat{U}_{1}^{*}, \mathbf{v}\big\rangle_{2}, &\qquad \widetilde{B} &= -\frac{\textnormal{i}}{4}\big\langle \hat{U}_{1}^{*}, \mathbf{u} - \textnormal{i}\mathbf{v}\big\rangle_{2},
    \end{aligned}
    \label{R-D:Transformation;Farfield}
\end{equation}
so that we can rewrite \eqref{R-D:Farf} as
\begin{equation}
	\begin{split}
     \frac{\textnormal{d}}{\textnormal{d} r} \widetilde{A} &= \textnormal{i}\widetilde{A} + \widetilde{B} - \frac{n}{2}\sigma\left(\widetilde{A} - \overline{\widetilde{A}}\right) - \frac{\mathrm{i}}{2}\mathcal{F}_{A}(\widetilde{A},\widetilde{B};\mu),\\
 \frac{\textnormal{d}}{\textnormal{d} r} 
 \widetilde{B} &= \mathrm{i}\widetilde{B} - \frac{n}{2}\sigma\left(\widetilde{B} + \overline{\widetilde{B}}\right) - \frac{1}{2}\mathcal{F}_{B}(\widetilde{A},\widetilde{B};\mu),\label{amp:AB;tilde} \\
    \frac{\textnormal{d}}{\textnormal{d} r} \sigma &= -\sigma^{2}.
	\end{split}
\end{equation}
Here we have defined $\mathcal{F}_{A} := \big\langle \hat{U}_{0}^{*} + \frac{1}{2}\hat{U}_{1}^{*}, \mathcal{F}\big\rangle$,  $\mathcal{F}_{B} := \big\langle \frac{1}{2}\hat{U}_{1}^{*}, \mathcal{F}\big\rangle$, and 
\begin{align}
    \mathcal{F} &= \mu \mathbf{M}_{2}\bigg[\big(\widetilde{A} + \overline{\widetilde{A}}\big) \hat{U}_{0} + 2 \textnormal{i} \big(\widetilde{B} - \overline{\widetilde{B}}\big)\hat{U}_{1}\bigg] + \mathbf{Q}_{0,0}\left( \widetilde{A}^2 + 2 |\widetilde{A}|^2 + \overline{\widetilde{A}}\,^2\right)  \nonumber\\
    &\quad +  4\textnormal{i} \mathbf{Q}_{0,1} \left( \widetilde{A}\widetilde{B} + \overline{\widetilde{A}} \widetilde{B} - \widetilde{A}\overline{\widetilde{B}} - \overline{\widetilde{A}}\overline{\widetilde{B}}\right) - 4 \mathbf{Q}_{1,1}\left(\widetilde{B}^2 - 2|\widetilde{B}|^2 +\overline{\widetilde{B}}\,^2 \right)\nonumber\\
      &\quad  + \mathbf{C}_{0,0,0} \left(\widetilde{A}^3 + 3  |\widetilde{A}|^2\widetilde{A} + 3  |\widetilde{A}|^2 \overline{\widetilde{A}}  + \overline{\widetilde{A}}\,^3\right) - 8\textnormal{i} \mathbf{C}_{1,1,1} \left(\widetilde{B}^3 - 3 |\widetilde{B}|^2\widetilde{B} +3 |\widetilde{B}|^2\overline{\widetilde{B}} -\overline{\widetilde{B}}\,^3\right)\nonumber\\
    &\quad  + 6\textnormal{i} \mathbf{C}_{0,0,1} \left(\widetilde{A}^2 \widetilde{B} + 2 |\widetilde{A}|^2 \widetilde{B} + \overline{\widetilde{A}}\,^2\widetilde{B} - \widetilde{A}^2 \overline{\widetilde{B}}- 2|\widetilde{A}|^2 \overline{\widetilde{B}} - \overline{\widetilde{A}}\,^2 \overline{\widetilde{B}}\right) \nonumber\\
    &\quad  - 12 \mathbf{C}_{0,1,1}\left(\widetilde{A}\widetilde{B}^2 -2\widetilde{A} |\widetilde{B}|^2 + \widetilde{A}\overline{\widetilde{B}}\,^2 + \overline{\widetilde{A}} \widetilde{B}^2 - 2 \overline{\widetilde{A}} |\widetilde{B}|^2 + \overline{\widetilde{A}} \overline{\widetilde{B}}\,^2\right). \nonumber
\end{align}
 where we recall from Hypothesis~\ref{R-D:hyp} that $\mathbf{Q}_{i,j} := \mathbf{Q}(\hat{U}_{i},\hat{U}_{j})$, $\mathbf{C}_{i,j,k} := \mathbf{C}(\hat{U}_{i},\hat{U}_{j},\hat{U}_{k})$ for $i,j,k\in\{0,1\}$.
\begin{rmk}\label{rmk:Core}
  The core manifold \eqref{Core:un} can be expressed as 
\begin{equation}\label{Core:An}
	\begin{split}
   \widetilde{A}(r_0) &= \frac{1}{2}r_0^{-\frac{n}{2}}\mathrm{e}^{\mathrm{i}y_n}\left[d_1 \left(1 + \mathcal{O}(r_0^{-1})\right) -\mathrm{i}d_2  r_0\left(1 + \mathcal{O}(r_0^{-1})\right) + \mathcal{O}_{r_0}(|\mu||\mathbf{d}| + |\mathbf{d}|^{2})\right]\\
   \widetilde{B}(r_0) &= \frac{1}{2}r_0^{-\frac{n}{2}}\mathrm{e}^{\mathrm{i}y_n}\left[ -\gamma\, d_1^2\left(\nu_n  + \mathcal{O}\left(r_0^{-\frac{n}{2}}+r_0^{-1}\right)\right) - \mathrm{i} d_2 \left(1 + \mathcal{O}(r_0^{-1})\right) + \mathcal{O}_{r_0}(|\mu||\mathbf{d}| + |d_2|^{2} + |d_1|^3)\right]\\
	\end{split}
\end{equation}
where $\nu_n = \left(\frac{3}{8}\right)^{\frac{n}{2}}\frac{\pi}{3\Gamma(\frac{n}{2})}$, $\gamma=\langle \hat{U}_1^*, \mathbf{Q}_{0,0} \rangle_{2}$ and $y_n = r_0 - \frac{n\pi}{4}$.
\end{rmk}

\begin{figure}
\,\hspace{0.5cm}\includegraphics[width=0.75\linewidth]{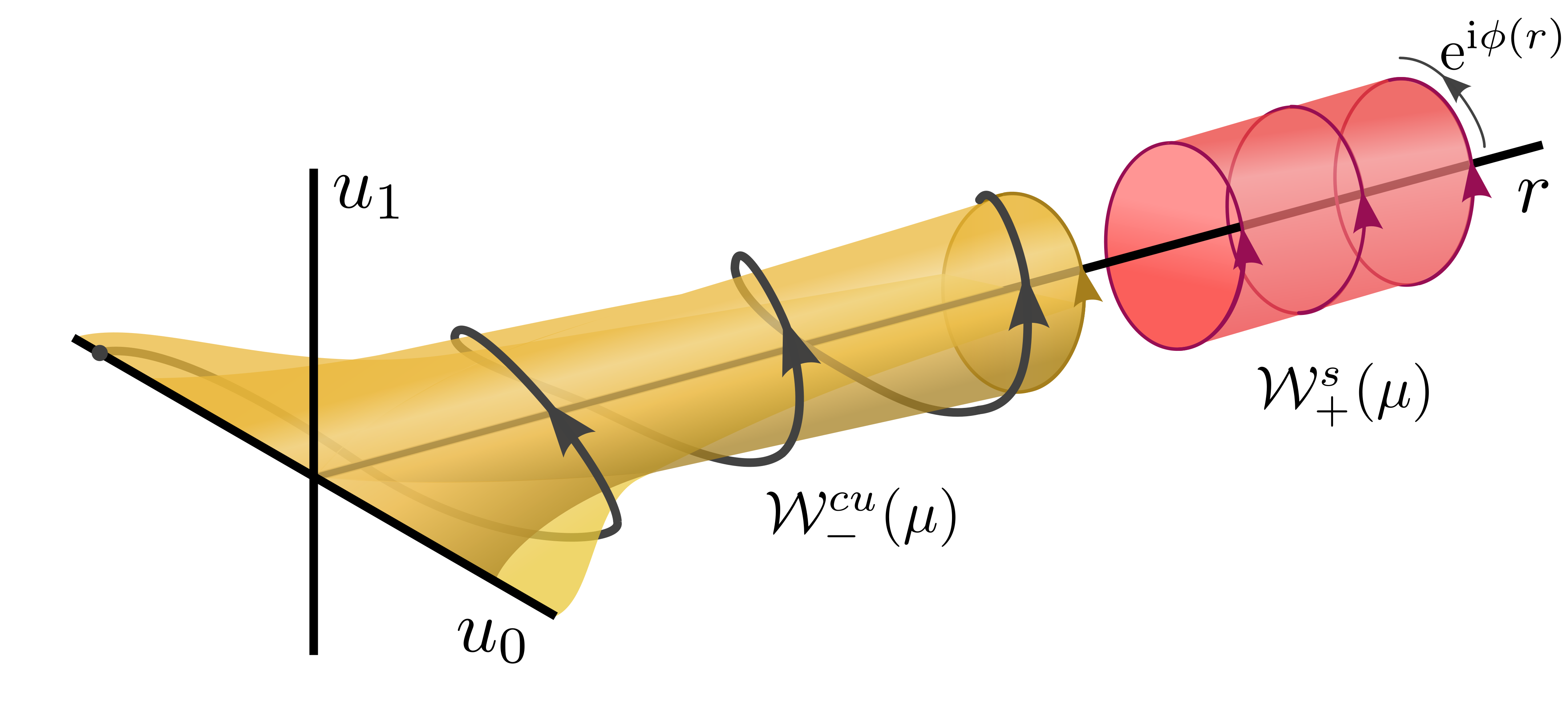}
    \caption{A diagram of the evolving shape of the core manifold $\mathcal{W}^{cu}_{-}(\mu)$ as $r$ increases, where $u_i:=\langle \hat{U}_{i}^*,\mathbf{u}\rangle$ for $i=0,1$. The far-field manifold $\mathcal{W}^{s}_{+}(\mu)$ is constructed in a co-rotating frame with phase $\mathrm{e}^{\mathrm{i}\phi(r)}$.}
    \label{fig:manifold}
\end{figure}

We note that the core manifold $\mathcal{W}^{cu}_{-}(\mu)$ converges to a co-rotating two-dimensional manifold \eqref{Core:An}, with a co-rotating phase $\mathrm{e}^{\mathrm{i}(r_0-\frac{n\pi}{4})}$. We then expect the far-field manifold to possess a co-rotating phase $\mathrm{e}^{\mathrm{i}\phi(r)}$, where $\phi(r)$ is some perturbation from $(r-\frac{n\pi}{4})$; see Figure~\ref{fig:manifold}. We apply standard normal form transformations, as seen in \cite{scheel2003radially,lloyd2009localized,Scheel2014Grain}, to \eqref{amp:AB;tilde} so that all low-order non-resonant terms and the co-rotating phase $\mathrm{e}^{\mathrm{i}\phi(r)}$ are removed from the right-hand side. For a more detailed proof of this result, see the analogous proof of Lemma 3.3 in Appendix A of \cite{hill2024dihedral}.

\begin{lem}\label{Lemma:normal} 
There exists a change of coordinates
\begin{equation}\label{Ntup:Normal;transf}
    \begin{pmatrix}
    A \\ B
    \end{pmatrix} := \textnormal{e}^{-\textnormal{i}\phi(r)}\left[\mathbbm{1} + \mathcal{T}(\sigma)\right]\begin{pmatrix}
    \widetilde{A} \\ \widetilde{B}
    \end{pmatrix} + \mathcal{O}((|\mu| + |\widetilde{A}| + |\widetilde{B}|)(|\widetilde{A}| + |\widetilde{B}|)),
\end{equation}
such that \eqref{amp:AB;tilde} becomes
\begin{equation}\label{Ntup:NormalForm}
	\begin{split}
    \frac{\textnormal{d}}{\textnormal{d}r} A &= - \frac{n}{2} \sigma A + B + R_{A}(A, B,\sigma,\mu),\\
    \frac{\textnormal{d}}{\textnormal{d}r} B &= -\frac{n}{2} \sigma B + c_{0}\,\mu A + c_3 |A|^2 A + R_{B}(A, B,\sigma,\mu),\\
    \frac{\textnormal{d}}{\textnormal{d} r} \sigma &= -\sigma^{2},
	\end{split}
\end{equation}
where we recall from Hypothesis~\ref{R-D:hyp} that $c_{0}= \frac{1}{4}\big\langle \hat{U}_{1}^{*}, -\mathbf{M}_{2}\hat{U}_{0}\big\rangle_{2}$ and
\begin{equation*}
    \begin{split}
    c_{3} ={}& - \bigg[\bigg(\frac{5}{6}\big[\big\langle \hat{U}_{0}^{*},\mathbf{Q}_{0,0}\big\rangle_{2} + \big\langle \hat{U}_{1}^{*},\mathbf{Q}_{0,1}\big\rangle_{2}\big] + \frac{19}{18}\big\langle \hat{U}_{1}^{*},\mathbf{Q}_{0,0}\big\rangle_{2} \bigg)\big\langle \hat{U}_{1}^{*},\mathbf{Q}_{0,0}\big\rangle_{2} +\frac{3}{4}\big\langle \hat{U}_{1}^{*}, \mathbf{C}_{0,0,0}\big\rangle_{2}\bigg]. \end{split}
\end{equation*}
The coordinate change \eqref{Ntup:Normal;transf} is polynomial in $(A,B,\sigma)$ and smooth in $\mu$, and $\mathcal{T}(\sigma) = \mathcal{O}(\sigma)$ is linear and upper triangular for each $\sigma$. The remainder terms satisfy
\begin{equation}
R_{A/B} = \mathcal{O}([|\mu|^{2} + |\sigma|^{3} + (|A| + |B|)^{2}]|B| + [|\mu|^{2} + |\sigma|^{3} + |A|^{3}]|A|),
\end{equation}
while $\phi(r)$ satisfies
\begin{equation}
    \frac{\textnormal{d}}{\textnormal{d} r}\phi = 1 + \mathcal{O}\left(|\mu| + |\sigma|^{2}\right), \qquad \phi(0) = 0.\nonumber
\end{equation}
\end{lem}

  The resultant system \eqref{Ntup:NormalForm} is the radial normal form for an $n$-dimensional Hamiltonian--Hopf bifurcation, which we extend further by introducing the unconstrained variable $\kappa(r)$ to take the place of $\sqrt{\mu}$. Then, \eqref{Ntup:NormalForm} becomes
\begin{equation}\label{Normal:Ext}
	\begin{split}
    	\frac{\textnormal{d}}{\textnormal{d}r} A &= - \frac{n}{2} \sigma A + B + R_{A}(A, B,\sigma,\kappa),\\
    \frac{\textnormal{d}}{\textnormal{d}r} B &= -\frac{n}{2} \sigma B + c_{0}\,\kappa^2 A + c_3 |A|^2 A + R_{B}(A, B,\sigma,\kappa),\\
    \frac{\textnormal{d}}{\textnormal{d} r} \sigma &= -\sigma^{2},\\
    		\frac{\textnormal{d}}{\textnormal{d}r}\kappa &= 0
	\end{split}
\end{equation}  
where we have
\begin{equation}
     	R_{A/B} = \mathcal{O}([\kappa^{4} + |\sigma|^{3} + (|A| + |B|)^{2}]|B| + [\kappa^{4} + |\sigma|^{3} + |A|^{3}]|A|). 
\end{equation}

The remainder of this section is focused on characterising exponentially decaying trajectories for the dynamical system \eqref{Normal:Ext}, restricted to the invariant subspace $\{(\sigma,\kappa) = (r^{-1}, \sqrt{\mu})\}$. We characterise such trajectories for $r \geq r_1 \mu^{-\frac{1}{2}}$ with $r_1\ll1$ fixed (the rescaling chart, see \S\,\ref{subsec:Rescaling}), and then for $r_0 \leq r \leq r_1 \mu^{-\frac{1}{2}}$ (the transition chart, see \S\,\ref{subsec:Transition}). Following this, we seek intersections between the far-field trajectories and the core manifold $\mathcal{W}^{cu}_{-}(\mu)$ via asymptotic matching at the point $r=r_0$ (see \S\,\ref{subsec:Matching}). Any trajectories that lie along these intersections will remain bounded as $r\to0^+$ while decaying exponentially fast as $r\to\infty$, thus characterising localised radial solutions to \eqref{eqn:R-D}.

The methods presented in the rest of this section are analogous to several previous studies of localised radial patterns, in particular the recent papers \cite{hill2023approximate,hill2024dihedral}, and so we present only condensed proofs of each result. For more detailed exposition on this approach, see \cite{Hill2021Localised,mccalla2013spots,lloyd2009localized,hill2023approximate,hill2024dihedral}.

\subsection{The Rescaling Chart}\label{subsec:Rescaling}
  To capture the `slow' behaviour of \eqref{Normal:Ext} for large values of $r$, we introduce the standard rescaling coordinates of \cite{mccalla2013spots}
\begin{equation}\label{resc}
    A_{R} :=  \kappa^{-1}A, \quad B_{R} := \kappa^{-2}B,\quad \sigma_{R} :=  \kappa^{-1}\sigma, \quad \kappa_{R} := \kappa, \quad s := \kappa r.
\end{equation}
Then, we can write \eqref{Normal:Ext} in the rescaling chart as,
\begin{equation}\label{Normal:Resc}
	\begin{split}
    		 \frac{\textnormal{d}}{\textnormal{d}s} A_{R} &= - \frac{n}{2} \sigma_{R} A_{R} + B_{R} + |\kappa_{R}|^{-2}R_{A}(\kappa_{R}A_{R}, \kappa^2_{R} B_{R},\kappa_{R}\sigma_{R},\kappa_{R}),\\
    \frac{\textnormal{d}}{\textnormal{d}s} B_{R} &= -\frac{n}{2} \sigma_{R} B_{R} + c_{0}\,A_{R} + c_3 |A_{R}|^2 A_{R} + |\kappa_{R}|^{-3}R_{B}(\kappa_{R}A_{R}, \kappa^2_{R} B_{R},\kappa_{R}\sigma_{R},\kappa_{R}),\\
    \frac{\textnormal{d}}{\textnormal{d} s} \sigma_{R} &= -\sigma_{R}^{2},\\
    		\frac{\textnormal{d}}{\textnormal{d}s}\kappa_{R} &= 0
	\end{split}
\end{equation}
where
\begin{equation}
\begin{split}
     |\kappa_{R}|^{-2} R_{A} ={}& |\kappa_{R}|^2\mathcal{O}([\kappa_{R}^{2} + |\kappa_{R}||\sigma_{R}|^{3} + (|A_{R}| + \kappa_{R}|B_{R}|)^{2}]|B_{R}| + [|\kappa_{R}| + |\sigma_{R}|^{3} + |A_{R}|^{3}]|A_{R}|),\\
     |\kappa_{R}|^{-3} R_{B} ={}& |\kappa_{R}|\mathcal{O}([\kappa_{R}^{2} + |\kappa_{R}||\sigma_{R}|^{3} + (|A_{R}| + \kappa_{R}|B_{R}|)^{2}]|B_{R}| + [|\kappa_{R}| + |\sigma_{R}|^{3} + |A_{R}|^{3}]|A_{R}|). 
\end{split}
\end{equation}
Evaluating \eqref{Normal:Resc} on the invariant subspace $\{\sigma_{R} = s^{-1},\; \kappa_{R} = 0\}$, we obtain the cubic nonautonomous complex Ginzburg--Landau equation
\begin{equation}\label{GL:comp}
    \mathcal{D}_{\frac{n}{2}}\mathcal{D}_{\frac{n}{2}} A_{R}(s) = c_0 A_{R}(s) + c_{3} |A_{R}(s)|^{2} A_{R}(s), \qquad A_{R}\in\mathbb{C},\qquad s\in[0,\infty),
\end{equation}
where, with a slight abuse of notation, we write $\mathcal{D}_{k}u(s)=(\frac{\mathrm{d}}{\mathrm{d}s} + \frac{k}{s})u(s)$ for any $k\in\mathbb{R}$. The equation \eqref{GL:comp} possesses $S^1$-symmetry, and so we reduce to the real subspace $A_{R}\in\mathbb{R}$ and consider solutions of the real Ginzburg--Landau equation
\begin{equation}\label{GL:real}
    \mathcal{D}_{\frac{n}{2}}\mathcal{D}_{\frac{n}{2}} A_{R}(s) = c_0 A_{R}(s) + c_{3} A_{R}(s)^{3}, \qquad A_{R}\in\mathbb{R}, \qquad s\in[0,\infty).
\end{equation}

We present the following lemma regarding the existence of a nondegenerate solution to \eqref{GL:real}, which is required for the existence of ring and spot B solutions.   We hence impose the restriction $0<n<4$, so that we can apply Proposition~\ref{prop:GL-Laplace} and Hypothesis~\ref{hyp:GL-Laplace}, as required for such solutions in Theorem~\ref{thm:SmallPatch}.
\begin{lem}\label{lem:GL}
Fix $0<n<4$, assume that $c_3<0$. Then, the Ginzburg--Landau equation \eqref{GL:real} has a forward-bounded localised solution $A_{R}(s)= \hat{q}(s) := |c_3|^{-\frac{1}{2}}\sqrt{c_0} q(\sqrt{c_0}s)$, with constants $q_n>0$, $p_{n}\neq0$, which depend only on $n$, such that
\begin{equation}\label{q:defn}
    q(s) = \begin{cases}
         q_n s^{\frac{2-n}{2}}\left(1 + \mathcal{O}(s)\right), & s\to 0,\\
         (p_n + \mathcal{O}(\mathrm{e}^{-s}))s^{-\frac{n}{2}}\mathrm{e}^{-s}, & s\to \infty.
    \end{cases}
\end{equation}
In addition, the linearisation of \eqref{GL:real} about $\hat{q}(s)$ does not possess a nontrivial solution that is forward-bounded on $[0,\infty)$. If $c_{3}>0$, then the only forward-bounded solution of \eqref{GL:real} on $[0,\infty)$ is $A_{R}(s)\equiv0$.
\end{lem}
\begin{proof}
We note that finding forward-bounded localised solutions to \eqref{GL:real} is equivalent to finding radially-symmetric ground state solutions to
    \begin{equation}\label{GL:Laplace}
        \Delta \hat{u} = \hat{u} + \mathrm{sgn}\,(c_3)|\mathbf{x}|^{2-n} \hat{u}^3, \qquad \hat{u}(\mathbf{x})\in\mathbb{R}, \qquad \mathbf{x}\in\mathbb{R}^{3}.
    \end{equation}
where $\hat{u}(\mathbf{x}) = |c_3|^{\frac{1}{2}} c_0^{\frac{n-4}{4}} |\mathbf{x}|^{\frac{n-2}{2}}A_{R}(c_{0}^{-\frac{1}{2}}|\mathbf{x}|)$. The key result here is that
\begin{equation*}
    \mathcal{D}_{\frac{n}{2}}\mathcal{D}_{\frac{n}{2}}\left(s^{\frac{2-n}{2}} u(s)\right) = s^{\frac{2-n}{2}} \mathcal{D}_{1}\mathcal{D}_{1}u(s)= s^{\frac{2-n}{2}} \Delta_{2} u(s)
\end{equation*}
which follows from the identities presented in \eqref{Dk:ident}. A standard variational argument shows that bounded localised solutions must also satisfy
\begin{equation*}
\begin{split}
    0 ={}& \int_{\mathbb{R}^{3}} \hat{u} \left(-\Delta \hat{u} + \hat{u} + \mathrm{sgn}\,(c_3)|\mathbf{x}|^{2-n} \hat{u}^3\right)\,\mathrm{d}\mathbf{x} = \int_{\mathbb{R}^{3}} \left\{|\nabla \hat{u}|^2 + |\hat{u}|^2 + \mathrm{sgn}\,(c_3)|\mathbf{x}|^{2-n} |\hat{u}|^4\right\}\,\mathrm{d}\mathbf{x},\\
\end{split}
\end{equation*}
and so we see that either $c_{3}<0$ or $\hat{u}\equiv0$. Hence the existence and non-degeneracy of $\hat{q}(s)$ follows from Proposition~\ref{prop:GL-Laplace} for $0<n<3$ and Hypothesis~\ref{hyp:GL-Laplace} for $3\leq n<4$. Finally, the asymptotic behaviour of $q(s)$ in \eqref{q:defn} can be found by noting that $q(s)$ is a solution of the rescaled equation
\begin{equation}\label{q:eqn}
    \mathcal{D}_{\frac{n}{2}}\mathcal{D}_{\frac{n}{2}} \hat{A}(s) = \hat{A}(s) - \hat{A}(s)^{3},
\end{equation}
for which we can write down a variation-of-constants formula and apply a standard fixed-point argument in each limit.
\end{proof}

We note that spot A solutions do not persist as $\mathcal{O}(\mu^{\frac{1}{2}})$ perturbations from the localised solution in Lemma~\ref{lem:GL}, but rather remain close to the linear flow of \eqref{GL:real}.

\begin{rmk}\label{rmk:GL-lin}
    For sufficiently small values of $A_{R}$, solutions of \eqref{GL:real} remain close to exponentially decaying solutions of the linear equation
\begin{equation}\label{GL:lin}
    \mathcal{D}_{\frac{n}{2}}\mathcal{D}_{\frac{n}{2}} A_{R}(s) = c_0 A_{R}(s), \qquad A_{R}\in\mathbb{R},\qquad s\in[0,\infty).
\end{equation}
    which have the explicit form $A_{R}(s) = s^{-\frac{n}{2}}\mathrm{e}^{-\sqrt{c_0}\,s}$.
\end{rmk}

We present the following lemma.
\begin{lem}\label{Lem:Resc;Evo} 
For each fixed choice of $r_{1}>0$, there are constants $a_0,\kappa_{0}>0$ such that there exist exponentially decaying solutions to \eqref{Normal:Resc} for $s\in[r_{1},\infty)$, with the following evaluations at $s=r_{1}$ for all $\mu<\kappa_{0}$, $|a|<a_0$. 
\begin{enumerate}[label=(\roman*)]
    \item  If $|A_{R}|=\mathcal{O}(a)$ as $|(a,\mu)|\to0^+$, then
\begin{equation}\label{Sol:SpotA;r1}
(A_{R}, B_{R}, \sigma_{R}, \kappa_{R})(r_{1}) = \bigg( a r_1^{-\frac{n}{2}}A_{0}, -\sqrt{c_0} a r_1^{-\frac{n}{2}}A_{0},r_1^{-1},\mu^{\frac{1}{2}}\bigg).
\end{equation}
\item If $c_3<0$, $n<4$ and $|A_{R}| = \mathcal{O}(1)$ as $|(a,\mu)|\to0^+$, then
\begin{equation}\label{Sol:Rings;r1}
(A_{R}, B_{R}, \sigma_{R}, \kappa_{R})(r_{1}) = \bigg( r_{1}^{\frac{2-n}{2}}\,A_{1}, r_{1}^{-\frac{n}{2}}\,A_{1},r_{1}^{-1}, \mu^{\frac{1}{2}}\bigg).
\end{equation}
Here we have defined
\begin{equation*}
    A_{0} := \textnormal{e}^{\textnormal{i}Y}(1 + \mathcal{O}(\mu^{\frac{1}{2}} + a^2)), \qquad  A_{1} := |c_3|^{-\frac{1}{2}}c_0^{\frac{4-n}{4}} q_n \,\textnormal{e}^{\textnormal{i}Y}(1 + \mathcal{O}(\mu^{\frac{1}{2}})),
\end{equation*}
and $Y\in\mathbb{R}$ is arbitrary.
\end{enumerate}
\end{lem}
\begin{proof}
    We first note that $\kappa_{R}$ acts as a parameter, and $\Gamma_{c}:=\{(\sigma_{R},\kappa_{R}) = (\tfrac{1}{s},\sqrt{c})\}$ is an invariant subspace of \eqref{Normal:Resc} for any fixed $c\geq0$. Evaluating \eqref{Normal:Resc} on $\Gamma_{0}$, we arrive at the system
\begin{equation*}
	\begin{split}
    		 \mathcal{D}_{\frac{n}{2}} A_{R}(s) &= B_{R}(s),\\
    \mathcal{D}_{\frac{n}{2}} B_{R}(s) &= c_{0}\,A_{R}(s) + c_3 |A_{R}(s)|^2 A_{R}(s),\\
	\end{split}
\end{equation*}
which reduces to the complex cubic non-autonomous Ginzburg--Landau equation \eqref{GL:comp}.
\begin{enumerate}[label=(\roman*)]
    \item Assuming $|A_{R}|=\mathcal{O}(|a|)$ with $|a|\ll1$ and recalling Remark~\ref{rmk:GL-lin}, the equation \eqref{GL:comp} possesses an exponentially decaying solution for $s\in[r_1,\infty)$ of the form
    \begin{equation}\label{Sol:spotA;Resc}
    \left(A_{R},B_{R},\sigma_{R},\kappa_{R}\right)(s) = \left(a s^{-\frac{n}{2}}\mathrm{e}^{\sqrt{c_0}(r_1 - s)}A_{0} , -\sqrt{c_0} a s^{-\frac{n}{2}}\mathrm{e}^{\sqrt{c_0}(r_1 - s)}A_{0}, s^{-1},\mu^{\frac{1}{2}}\right),
\end{equation}
    for some $a\in\mathbb{R}$, $|a|\ll1$, where the $\mathcal{O}(\mu^{\frac{1}{2}} + a^2)$ terms are due to the $\mathcal{O}(|\kappa_{R}|)$ and $\mathcal{O}(|A_{R}|^2 A_{R})$ terms in \eqref{Normal:Resc}.

\item Assuming that $c_{3}<0$ and $n<4$, it follows from Lemma~\ref{lem:GL} that the equation \eqref{GL:comp} possesses an exponentially decaying solutions for $s\in[r_1,\infty)$ of the form
\begin{equation}\label{Sol:ring;Resc}
    \left(A_{R},B_{R},\sigma_{R},\kappa_{R}\right)(s) = \left(\hat{q}(s)\,\mathrm{e}^{\mathrm{i}Y}(1 + \mathcal{O}(\mu^{\frac{1}{2}})),\mathcal{D}_{\frac{n}{2}}\hat{q}(s)\,\mathrm{e}^{\mathrm{i}Y}(1 + \mathcal{O}(\mu^{\frac{1}{2}})),\tfrac{1}{s},\mu^{\frac{1}{2}}\right),
\end{equation}
where the $\mathcal{O}(\mu^{\frac{1}{2}})$ terms are due to the $\mathcal{O}(|\kappa_{R}|)$ terms in \eqref{Normal:Resc}. 
\end{enumerate}
Evaluating \eqref{Sol:spotA;Resc} and \eqref{Sol:ring;Resc} at $s=r_1\ll1$ and (in the case of $(ii)$) applying the asymptotic properties of $\hat{q}(s)=|c_3|^{-\frac{1}{2}}\sqrt{c_0} q(\sqrt{c_0}s)$ from \eqref{q:defn}, we arrive at our final result.
\end{proof}
  After inverting the rescaling transformation \eqref{resc}, exponentially decaying solutions to \eqref{Normal:Ext} evaluated at $r=r_{1}\mu^{-\frac{1}{2}}$ then either take the form
\begin{equation}\label{Sol:SpotA;delta0} 
    (A, B, \sigma, \kappa)(r_{1}\mu^{-\frac{1}{2}}) = \bigg( a \mu^{\frac{1}{2}} r_1^{-\frac{n}{2}}A_{0}, -\sqrt{c_0} a \mu r_1^{-\frac{n}{2}}A_{0}, \mu^{\frac{1}{2}} r_1^{-1},\mu^{\frac{1}{2}}\bigg),  
\end{equation}
or, if $c_3<0$ and $n<4$,
\begin{equation}\label{Sol:Rings;delta0} 
    (A, B, \sigma, \kappa)(r_{1}\mu^{-\frac{1}{2}}) = \bigg( \mu^{\frac{1}{2}} r_{1}^{\frac{2-n}{2}}\,A_{1}, \mu r_{1}^{-\frac{n}{2}}\,A_{1}, \mu^{\frac{1}{2}} r_{1}^{-1}, \mu^{\frac{1}{2}}\bigg),  
\end{equation}
for $a,\mu$ sufficiently small.

\subsection{The Transition Chart}\label{subsec:Transition}

  So far, we have identified exponentially decaying trajectories in \eqref{Normal:Ext} for $r>r_{1}\mu^{-\frac{1}{2}}$. To connect these trajectories with the core manifold constructed in \S\ref{subsec:Core}, we must track solutions with boundary data \eqref{Sol:SpotA;delta0} and \eqref{Sol:Rings;delta0} backwards through the `transition' region $r_{0}\leq r\leq r_{1}\mu^{-\frac{1}{2}}$ to the boundary of the core manifold at $r=r_0$.

\begin{lem} 
For each fixed choice of $0<r_{1}, r_{0}^{-1}\ll1$, there are constants $a_0,\kappa_{0}>0$ such that solutions of \eqref{Normal:Ext}, evaluated at $r=r_{0}$, $\kappa(r_0)=\mu^{\frac{1}{2}}$, for initial values \textnormal{(\ref{Sol:SpotA;delta0}-\ref{Sol:Rings;delta0})} are given by the following forms for all $\mu<\kappa_{0}$, $|a|<a_0$ and $Y\in\mathbb{R}$.
\begin{enumerate}[label=(\roman*)]
    \item For initial value \eqref{Sol:SpotA;delta0}, we obtain a solution of the form
    \begin{equation}\label{Sol:SpotA;r0}
	\begin{split}
   		A(r_{0}) &= \mu^{\frac{2-n}{4}} r_0^{-\frac{n}{2}}a\mathrm{e}^{\mathrm{i}Y}(1 + \xi_0),\\
      		B(r_{0}) &= -\sqrt{c_0} \mu^{\frac{4-n}{4}}r_0^{-\frac{n}{2}}a\left[1+\frac{c_3}{\sqrt{c_0}}\frac{a^2}{\mathcal{E}^2_n(\mu)} \right]  \mathrm{e}^{\mathrm{i}Y}(1 + \xi_0). \\
	\end{split}
\end{equation}
\item For initial value \eqref{Sol:Rings;delta0} with $c_3<0$ and $n<4$, we obtain a solution of the form
\begin{equation}\label{Sol:Rings;r0}
	\begin{split}
   		A(r_{0}) &= |c_3|^{-\frac{1}{2}} c_0^{\frac{4-n}{4}} q_n\mu^{\frac{4-n}{4}} r_0^{\frac{2-n}{2}}\textnormal{e}^{\textnormal{i}Y}(1 + \xi_1),\\
      		B(r_{0}) &= |c_3|^{-\frac{1}{2}} c_0^{\frac{4-n}{4}} q_n\mu^{\frac{4-n}{4}}r_0^{-\frac{n}{2}}\textnormal{e}^{\textnormal{i}Y}(1 + \xi_1). \\
	\end{split}
\end{equation}
\item Additionally, for initial value \eqref{Sol:Rings;delta0} with $c_3<0$ and $n<4$, we also obtain a solution of the form
\begin{equation}\label{Sol:SpotB;r0}
	\begin{split}
   		A(r_{0}) &= a |c_3|^{-\frac{1}{2}} c_0^{\frac{4-n}{4}} q_n\mu^{\frac{4-n}{8}} r_{0}^{-\frac{n}{2}} \textnormal{e}^{\textnormal{i}Y}(1 + \xi_{2}),\\
      		B(r_{0}) &= |c_3|^{-\frac{1}{2}} c_0^{\frac{4-n}{4}} q_n\mu^{\frac{4-n}{4}}r_{0}^{-\frac{n}{2}} \textnormal{e}^{\textnormal{i}Y}(1 + \xi_{2}). \\
	\end{split}
\end{equation}
\end{enumerate}
Here we have defined the nonlinear function
\begin{equation*}
    \mathcal{E}_{n}(\mu):
    = \mu^{\frac{n-1}{4}}\left(\frac{(n-1)(\mu^{-\frac{1}{2}}r_1)^{n-1} r_0^{n-1}}{\left(\mu^{-\frac{1}{2}} r_1\right)^{n-1} - r_0^{n-1}}\right)^{\frac{1}{2}}
    = \begin{cases}
        \mathcal{O}_{r_0,r_1}(1), & 0<n<1,\\
        \mathcal{O}_{r_0,r_1}(|\log(\mu)|^{-\frac{1}{2}}), & n=1,\\
        \mathcal{O}_{r_0,r_1}(\mu^{\frac{n-1}{4}}), & n>1,\\
    \end{cases}
\end{equation*}
and introduced
\begin{equation*}
    \xi_0 := \mathcal{O}(\mu^{\frac{1}{2}} + a^2 + r_1 + r_0^{-1}), \qquad  \xi_{1} :=  \mathcal{O}(\mu^{\frac{1}{2}} + r_1 + r_0^{-1}), \qquad  \xi_{2} :=  \mathcal{O}(\mu^{\frac{n}{2}} + \mu^{\frac{4-n}{8}} + |a| + r_1 + r_0^{-1}),
\end{equation*}
for the remainder terms in \textnormal{(\ref{Sol:SpotA;r0}-\ref{Sol:SpotB;r0})}.
\label{Lem:Transition;Evo}
\end{lem}
\begin{proof}
  We note that the equations for $\kappa$ and $\sigma$ in \eqref{Normal:Ext} decouple from $A$ and $B$, and so we begin by solving \eqref{Normal:Ext} for $\kappa$ and $\sigma$ with initial data \textnormal{(\ref{Sol:SpotA;delta0}-\ref{Sol:Rings;delta0})}, giving
\begin{equation*}
    \kappa(r) = \kappa\left(r_{1}\mu^{-\frac{1}{2}}\right) = \mu^{\frac{1}{2}}, \quad \quad \sigma(r) = \frac{1}{\left(r-r_{1}\mu^{-\frac{1}{2}}\right) + \sigma\left(r_{1}\mu^{-\frac{1}{2}}\right)^{-1}} = \frac{1}{r}.
\end{equation*}
Our approach is then as follows. Given some fixed choice of $R_{-}, R_{+}>0$ with  $R_{-}<R_{+}$, we integrate backwards over $R_{-}\leq r\leq R_{+}$, such that \eqref{Normal:Ext} (evaluated at $\kappa=\mu^{\frac{1}{2}}$, $\sigma = \frac{1}{r}$) then becomes the integral equation
\begin{equation}\label{Normal:Transition;R+-}
	\begin{split}
    	A(r) &= A(R_{+}) +  \int_{R_{+}}^{r} \left\{- \frac{n}{2p} A(p) + B(p) + R_{A}(A(p), B(p), p^{-1}, \mu^{\frac{1}{2}})\right\}\,\mathrm{d}p,\\
    B(r) &= B(R_{+}) + \int_{R_{+}}^{r} \left\{ -\frac{n}{2p} B(p) + c_{0}\,\mu A(p) + c_3 |A(p)|^2 A(p) + R_{B}(A(p), B(p),p^{-1},\mu^{\frac{1}{2}})\right\}\,\mathrm{d}p.\\
	\end{split}
\end{equation}  
For sufficiently small values of $(A(R_+), B(R_+))$, we can apply the contraction mapping principle to show that \eqref{Normal:Transition;R+-} has a unique solution $(A,B)$ in an appropriate small ball centred at the origin in $C([R_{-},R_{+}], \mathbb{C}^{2})$; see, for example, \cite{Sandstede1997Convergence}. Furthermore, we can express the unique solution to \eqref{Normal:Transition;R+-} evaluated at the point $r=R_{-}$ as a perturbation from $(A(R_{+}), B(R_{+}))$. \par\,

For each initial value \textnormal{(\ref{Sol:SpotA;delta0}, \ref{Sol:Rings;delta0})}, we utilise this approach in conjunction with appropriate choices of $R_-$, $R_+$ and coordinate transformations for $(A,B)$ in order to obtain our expressions for $(A(r_0), B(r_0))$.
\begin{enumerate}[label=(\roman*)]
    \item We begin by considering solutions with initial data \eqref{Sol:SpotA;delta0}. We introduce the following transition coordinates
\begin{align}
    A_{T} := r^{\frac{n}{2}}A, \qquad B_{T} := r^{\frac{n}{2}}B, \label{SpotA:Transition}
\end{align}
and integrate over $r_{0}\leq r\leq r_{1}\mu^{-\frac{1}{2}}$, so that \eqref{Normal:Transition;R+-} is replaced by
\begin{equation*}
	\begin{split}
 A_T(r)  &= a \mu^{\frac{2-n}{4}} A_{0} + \int_{r_1\mu^{-\frac{1}{2}}}^{r} \left\{B_{T}(p) + R_{A,T}(A_{T}, B_{T};p,\mu^{\frac{1}{2}})\right\}\,\mathrm{d}p,\\
    B_T(r) &= -\sqrt{c_0} a \mu^{\frac{4-n}{4}} A_{0} + \int_{r_1\mu^{-\frac{1}{2}}}^{r} \left\{c_{0}\,\mu A_{T}(p) + c_3 p^{-n} |A_{T}(p)|^2 A_{T}(p) + R_{B,T}(A_{T}, B_{T};p,\mu^{\frac{1}{2}})\right\}\,\mathrm{d}p,\\
	\end{split}
\end{equation*}
where 
\begin{equation*}\begin{split}
    R_{A,T}(A_T,B_T;r,\mu^{\frac{1}{2}}) ={}& r^{\frac{n}{2}} R_{A}(r^{-\frac{n}{2}}A_T, r^{-\frac{n}{2}}B_T,r^{-1},\mu^{\frac{1}{2}}), \\
    ={}& \mathcal{O}([\mu^{2} + r^{-3} + r^{-n}(|A_T| + |B_T|)^{2}]|B_T| + [\mu^{2} + r^{-3} + r^{-\frac{3n}{2}}|A_T|^{3}]|A_T|), \\
    R_{B,T}(A_T,B_T;r,\mu^{\frac{1}{2}}) ={}& r^{\frac{n}{2}} R_{B}(r^{-\frac{n}{2}}A_T, r^{-\frac{n}{2}}B_T,r^{-1},\mu^{\frac{1}{2}}),\\
    ={}& \mathcal{O}([\mu^{2} + r^{-3} + r^{-n}(|A_T| + |B_T|)^{2}]|B_T| + [\mu^{2} + r^{-3} + r^{-\frac{3n}{2}}|A_T|^{3}]|A_T|).
\end{split}
\end{equation*}
For sufficiently small values of $\mu$, we obtain a unique solution $(A_T, B_T)$ whose evaluation at the point $r=r_0$ has the following form
\begin{equation}\label{Transition:SpotA;Initial}
    A_{T}(r_{0}) = \mu^{\frac{2-n}{4}} a\mathrm{e}^{\mathrm{i}Y}(1 + \xi_0), \quad B_{T}(r_{0}) = -\sqrt{c_0} \mu^{\frac{4-n}{4}} a\left[1+\frac{c_3}{\sqrt{c_0}}\frac{a^2}{\mathcal{E}_{n}(\mu)^2} \right]  \mathrm{e}^{\mathrm{i}Y}(1 + \xi_0).
\end{equation}
where we define $\mathcal{E}_{n}(\mu)$ such that
\begin{equation*}
    \frac{1}{\mathcal{E}_{n}(\mu)^{2}} = \mu^{\frac{1-n}{2}}\int_{r_1\mu^{-\frac{1}{2}}}^{r_0} p^{-n} \,\mathrm{d}p = \frac{(r_1\mu^{-\frac{1}{2}})^{n-1}-r_0^{n-1}}{(n-1)(r_0 r_1)^{n-1}}.
\end{equation*}
\item We now consider solutions with initial values \eqref{Sol:Rings;delta0}. We introduce the following transition coordinates
\begin{align}
    \widetilde{A}_{T} := r^{\frac{n-2}{2}}A, \qquad \widetilde{B}_{T} := r^{\frac{n}{2}}B, \label{Rings;Transition}
\end{align}
and integrate over $r_0\leq r\leq r_1 \mu^{-\frac{1}{2}}$, so that \eqref{Normal:Transition;R+-} is replaced by
\begin{equation*}
	\begin{split}
 \widetilde{A}_T(r)  &= \mu^{\frac{4-n}{4}}A_1 + \int_{r_1\mu^{-\frac{1}{2}}}^{r} \left\{- p^{-1} \widetilde{A}_T(p) +  p^{-1} \widetilde{B}_T(p) + R_{A,T}(\widetilde{A}_T, \widetilde{B}_T;p,\mu^{\frac{1}{2}})\right\}\,\mathrm{d}p,\\
    \widetilde{B}_T(r) &= \mu^{\frac{4-n}{4}}A_1 + \int_{r_1\mu^{-\frac{1}{2}}}^{r} \left\{c_{0}\,\mu p \widetilde{A}_T(p) + c_3 p^{3-n} |\widetilde{A}_T(p)|^2 \widetilde{A}_T(p) + R_{B,T}(\widetilde{A}_T, \widetilde{B}_T;p,\mu^{\frac{1}{2}})\right\}\,\mathrm{d}p.\\
	\end{split}
\end{equation*}
where 
\begin{equation*}\begin{split}
    R_{A,T}(\widetilde{A}_T,\widetilde{B}_T;r,\mu^{\frac{1}{2}}) ={}& r^{\frac{n-2}{2}} R_{A}(r^{\frac{2-n}{2}}\widetilde{A}_T, r^{-\frac{n}{2}}\widetilde{B}_T,r^{-1},\mu^{\frac{1}{2}}), \\
    ={}& \mathcal{O}(r^{-1}[\mu^{2} + r^{-3} + r^{-n}(r|\widetilde{A}_T| + |\widetilde{B}_T|)^{2}]|\widetilde{B}_T| + [\mu^{2} + r^{-3} + r^{\frac{3(2-n)}{2}}|\widetilde{A}_T|^{3}]|\widetilde{A}_T|), \\
    R_{B,T}(\widetilde{A}_T,\widetilde{B}_T;r,\mu^{\frac{1}{2}}) ={}& r^{\frac{n}{2}} R_{B}(r^{\frac{2-n}{2}}\widetilde{A}_T, r^{-\frac{n}{2}}\widetilde{B}_T,r^{-1},\mu^{\frac{1}{2}}),\\
    ={}& \mathcal{O}([\mu^{2} + r^{-3} + r^{-n}(r|\widetilde{A}_T| + |\widetilde{B}_T|)^{2}]|\widetilde{B}_T| + r[\mu^{2} + r^{-3} + r^{\frac{3(2-n)}{2}}|\widetilde{A}_T|^{3}]|\widetilde{A}_T|).
\end{split}
\end{equation*}
For sufficiently small values of $\mu$, we obtain a unique solution $(\widetilde{A}_T, \widetilde{B}_T)$ whose evaluation at the point $r=r_0$ has the following form
\begin{equation}\label{Transition:Rings;Initial}
    \widetilde{A}_{T}(r_{0}) = |c_3|^{-\frac{1}{2}} c_0^{\frac{4-n}{4}} q_n\mu^{\frac{4-n}{4}} \textnormal{e}^{\textnormal{i}Y}(1 + \xi_{1}), \quad \widetilde{B}_{T}(r_{0}) = |c_3|^{-\frac{1}{2}} c_0^{\frac{4-n}{4}} q_n\mu^{\frac{4-n}{4}}\textnormal{e}^{\textnormal{i}Y}(1 + \xi_{1}).
\end{equation}
\item Finally, we consider solutions with initial values \eqref{Sol:Rings;delta0} which transition between $(\widetilde{A}_{T}, \widetilde{B}_{T})$ and $(A_{T}, B_{T})$ coordinates during their evolution in $r$, as observed in \cite{mccalla2013spots}. To do this we introduce some $a>0$ such that $r_0 < a\mu^{-\frac{4-n}{8}} < r_1 \mu^{-\frac{1}{2}}$. We then consider the transition coordinates $(\widetilde{A}_{T}, \widetilde{B}_{T})$ as defined in \eqref{Rings;Transition} and integrate over $a\mu^{-\frac{4-n}{8}}\leq r\leq r_{1}\mu^{-\frac{1}{2}}$, so that \eqref{Normal:Transition;R+-} is replaced by
\begin{equation*}
	\begin{split}
 \widetilde{A}_T(r)  &= \mu^{\frac{4-n}{4}}A_1 + \int_{r_1\mu^{-\frac{1}{2}}}^{r} \left\{- p^{-1} \widetilde{A}_T(p) +  p^{-1} \widetilde{B}_T(p) + R_{A,T}(\widetilde{A}_T, \widetilde{B}_T;p,\mu^{\frac{1}{2}})\right\}\,\mathrm{d}p,\\
    \widetilde{B}_T(r) &= \mu^{\frac{4-n}{4}}A_1 + \int_{r_1\mu^{-\frac{1}{2}}}^{r} \left\{c_{0}\,\mu p \widetilde{A}_T(p) + c_3 p^{3-n} |\widetilde{A}_T(p)|^2 \widetilde{A}_T(p) + R_{B,T}(\widetilde{A}_T, \widetilde{B}_T;p,\mu^{\frac{1}{2}})\right\}\,\mathrm{d}p.\\
	\end{split}
\end{equation*}
As we saw previously, for sufficiently small values of $\mu$ we obtain a unique solution for $(\widetilde{A}_{T}, \widetilde{B}_{T})$ whose evaluation at the point $r=a\mu^{-\frac{4-n}{8}}$ has the following form
\begin{equation}
    \widetilde{A}_{T}(a\mu^{-\frac{4-n}{8}}) = \mu^{\frac{4-n}{4}} \widetilde{A}_1, \quad \widetilde{B}_{T}(a\mu^{-\frac{4-n}{8}}) = \mu^{\frac{4-n}{4}}\widetilde{A}_{1},
\end{equation}
where we have defined $\widetilde{A}_1:=|c_3|^{-\frac{1}{2}}c_0^{\frac{4-n}{4}} q_n \,\textnormal{e}^{\textnormal{i}Y}(1 + \mathcal{O}(\mu^{\frac{4-n}{8}} + r_1))$. Transforming into the other transition coordinates $(A_T, B_T)$, defined in \eqref{SpotA:Transition}, and integrating over $r_0 \leq r \leq a \mu^{-\frac{4-n}{8}}$, \eqref{Normal:Transition;R+-} is again replaced by
\begin{equation*}
	\begin{split}
 A_T(r)  &= a\mu^{\frac{4-n}{8}} \widetilde{A}_1 + \int_{a \mu^{-\frac{4-n}{8}}}^{r} \left\{B_{T}(p) + R_{A,T}(A_{T}, B_{T};p,\mu^{\frac{1}{2}})\right\}\,\mathrm{d}p,\\
    B_T(r) &= \mu^{\frac{4-n}{4}} \widetilde{A}_1 + \int_{a \mu^{-\frac{4-n}{8}}}^{r} \left\{c_{0}\,\mu A_{T}(p) + c_3 p^{-n} |A_{T}(p)|^2 A_{T}(p) + R_{B,T}(A_{T}, B_{T};p,\mu^{\frac{1}{2}})\right\}\,\mathrm{d}p,\\
	\end{split}
\end{equation*}

Finally, for sufficiently small values of $\mu$ we obtain a unique solution for $(A_{T}, B_{T})$ whose evaluation at the point $r=r_0$ has the following form
\begin{equation}\label{Transition:SpotB;Initial}
    A_{T}(r_0) = a |c_3|^{-\frac{1}{2}} c_0^{\frac{4-n}{4}} q_n\mu^{\frac{4-n}{8}} \textnormal{e}^{\textnormal{i}Y}(1 + \xi_{2}), \quad B_{T}(r_0) = |c_3|^{-\frac{1}{2}} c_0^{\frac{4-n}{4}} q_n\mu^{\frac{4-n}{4}}\textnormal{e}^{\textnormal{i}Y}(1 + \xi_{2}).
\end{equation}
\end{enumerate}
Inverting the transformations \textnormal{(\ref{SpotA:Transition}, \ref{Rings;Transition})} brings us to the desired result.
\end{proof} 

\subsection{Matching Core and Far-field Manifolds}\label{subsec:Matching}
In the previous subsection we have obtained a parametrisation for exponentially-decaying trajectories of \eqref{R-D:Farf}, which lie on the far-field manifold $\mathcal{W}^{s}_{+}(\mu)$, evaluated at the fixed point $r=r_0$. We now seek intersections of these trajectories with the parametrisation of the core manifold $\mathcal{W}^{cu}_{-}(\mu)$, evaluated at $r=r_0$, defined in \eqref{Core:un}. We recall from Remark~\ref{rmk:Core} that the core manifold $\mathcal{W}^{cu}_{-}(\mu)$, following the normal form transformations \eqref{Ntup:Normal;transf} can be written as
\begin{equation}\label{Core:Am;param}
	\begin{split}
   \widetilde{A}(r_0) &= \frac{1}{2}r_0^{-\frac{n}{2}}\mathrm{e}^{\mathrm{i}Y_n}\left[d_1 \left(1 + \mathcal{O}(r_0^{-1})\right) -\mathrm{i}d_2  r_0\left(1 + \mathcal{O}(r_0^{-1})\right) + \mathcal{O}_{r_0}(|\mu||\mathbf{d}| + |\mathbf{d}|^{2})\right]\\
   \widetilde{B}(r_0) &= -\frac{1}{2}r_0^{-\frac{n}{2}}\mathrm{e}^{\mathrm{i}Y_n}\left[ \gamma\, d_1^2\left(\nu_n + \mathcal{O}\left(r_0^{-\frac{n}{2}}+r_0^{-1}\right)\right) +\mathrm{i} d_2 \left(1 + \mathcal{O}(r_0^{-1})\right) + \mathcal{O}_{r_0}(|\mu||\mathbf{d}| + |d_2|^{2} + |d_1|^3)\right]\\
	\end{split}
\end{equation}
where $\nu_n = \left(\frac{3}{8}\right)^{\frac{n}{2}}\frac{\pi}{3\Gamma(\frac{n}{2})}$, $\gamma=\langle \hat{U}_1^*, \mathbf{Q}(\hat{U}_0,\hat{U}_0) \rangle_{2}$ and $Y_n = - \frac{n\pi}{4}+ \mathcal{O}(|\mu|r_{0} + r_{0}^{-2})$.

\begin{lem}\label{Lem:Match}
    For each fixed choice of $0<r_1, r_0^{-1}\ll1$, there is a constant $\kappa_0>0$ such that intersections between the core manifold \eqref{Core:Am;param} and respective far-field trajectories \textnormal{(\ref{Sol:SpotA;r0}-\ref{Sol:SpotB;r0})} satisfy the following conditions for all $\mu<\kappa_0$.
    \begin{enumerate}[label=(\roman*)]
        \item For \eqref{Sol:SpotA;r0} with $\gamma\neq0$, we obtain $Y = Y_{n}$ and
        \begin{equation}\label{Match:SpotA}
            d_1 = (c_0 \mu)^{\frac{1}{2}}\frac{1}{\nu_n \gamma}\left(1 + \mathcal{O}\left(r_0^{-1} + r_1 + \mu^{\frac{1}{2}} + \mu^{\frac{n}{2}}\right)\right)\qquad d_2 = \mu\, \mathcal{O}\left(r_0^{-1} + r_1 + \mu^{\frac{1}{2}} + \mu^{\frac{n}{2}}\right).
        \end{equation}
                \item For \eqref{Sol:SpotA;r0} with $\gamma = \mu^{\frac{n}{4}}\mathcal{E}_{n}(\mu)^{-1} \tilde{\gamma}$, we obtain $Y = Y_{n}$ and
    \begin{equation}\label{Match:SpotA;fold}
            d_1 = \frac{2\left[\nu_n \widetilde{\gamma}\,\pm\sqrt{\nu_n^2 \widetilde{\gamma}^2 - \sqrt{c_0}c_3}\right]}{c_3}\,\mu^{\frac{2-n}{4}}\mathcal{E}_{n}(\mu)\left(1 + \mathrm{o}(1)\right), \qquad d_2 = \mu^{\frac{4-n}{4}}\mathcal{E}_{n}(\mu)\,\mathrm{o}(1).
        \end{equation}
        \item For \eqref{Sol:Rings;r0} with $c_3<0$ and $n<4$, we obtain $Y = Y_{n} + (2\pm1)\frac{\pi}{2}$ and
        \begin{equation}\label{Match:Rings}
            d_1 = \mu^{\frac{4-n}{4}}\mathcal{O}\left(r_0^{-1} + r_1 + \mu^{\frac{1}{2}}+ \mu^{\frac{4-n}{4}}\right),\qquad d_2 = \pm \frac{2 q_n}{\sqrt{|c_3|}} (c_0 \mu)^{\frac{4-n}{4}}\left(1 + \mathcal{O}\left(r_0^{-1} + r_1 + \mu^{\frac{1}{2}}+ \mu^{\frac{4-n}{4}}\right)\right).
        \end{equation}
        \item For \eqref{Sol:SpotB;r0} with $c_3<0$ and $n<4$, we obtain $Y = Y_{n} + (\mathrm{sgn}(\gamma)-1)\frac{\pi}{2}$ and
        \begin{equation}\label{Match:SpotB}
            d_1 = -\mathrm{sgn}(\gamma)\sqrt{\frac{2q_n}{\nu_n |\gamma|\sqrt{|c_3|}}} (c_0\mu)^{\frac{4-n}{8}}\,\left(1 + \mathcal{O}(r_{0}^{-1} + r_{1} + \mu^{\frac{4-n}{8}})\right),\qquad d_2 = \mu^{\frac{4-n}{4}}\mathcal{O}\left(r_{0}^{-1} + r_{1} + \mu^{\frac{4-n}{8}}\right).
        \end{equation}
    \end{enumerate}
\end{lem}
\begin{proof}
We consider each case separately:
    \begin{enumerate}[label=(\roman*)]
        \item Setting the far-field trajectory \eqref{Sol:SpotA;r0} and the core parametrisation \eqref{Core:Am;param} equal to each other, we look to solve
\begin{equation}\label{match:SpotA;corefar}
	\begin{split}
    		2\mu^{\frac{2-n}{4}} a\mathrm{e}^{\mathrm{i}(Y-Y_n)}(1 + \xi_0) ={}& d_1 \left(1 + \mathcal{O}(r_0^{-1})\right) -\mathrm{i}d_2  r_0\left(1 + \mathcal{O}(r_0^{-1})\right) + \mathcal{O}_{r_0}(|\mu||\mathbf{d}| + |\mathbf{d}|^{2}),\\
    		2\sqrt{c_0}\, g_n(a,\mu)\mathrm{e}^{\mathrm{i}(Y-Y_n)}(1 + \xi_0) ={}& \gamma\, d_1^2\left(\nu_n + \mathcal{O}\left(r_0^{-\frac{n}{2}}+r_0^{-1}\right)\right) + \mathrm{i} d_2 \left(1 + \mathcal{O}(r_0^{-1})\right)\\
      &\hspace{0.25\linewidth} + \mathcal{O}_{r_0}(|\mu||\mathbf{d}| + |d_2|^{2} + |d_1|^3), 
	\end{split}
\end{equation}
where $g_n(a,\mu) := \mu^{\frac{4-n}{4}} a(1 + \frac{c_3}{\sqrt{c_0}}\mathcal{E}_{n}(\mu)^{-2} a^2)$. We introduce the following coordinate transformations 
\begin{equation}\label{match:SpotA;scaling}
    	d_{1} = \sqrt{c_0}\mu^{\frac{1}{2}}\widetilde{d}_{1},\qquad d_{2} = c_0\mu\widetilde{d}_{2},\qquad a = \frac{\sqrt{c_0}}{2}\mu^{\frac{n}{4}}\widetilde{a},\qquad  Y = Y_{n} + \widetilde{Y},
\end{equation}
so that \eqref{match:SpotA;corefar} becomes
\begin{equation}\label{match:SpotA;corefar-scale}
	\begin{split}
    \widetilde{a}\mathrm{e}^{\mathrm{i}\widetilde{Y}}(1 + \xi_0) ={}& \widetilde{d}_{1} \left(1 + \mathcal{O}(r_0^{-1})\right) + \mathcal{O}_{r_0}(\mu^{\frac{1}{2}}),\\
    		\widetilde{a}\mathrm{e}^{\mathrm{i}\widetilde{Y}}(1 + \xi_0) ={}& \gamma\, \widetilde{d}_{1}^2\left(\nu_n + \mathcal{O}\left(r_0^{-\frac{n}{2}}+r_0^{-1}\right)\right)  + \mathrm{i}\widetilde{d}_{2} \left(1 + \mathcal{O}(r_0^{-1})\right)  + \mathcal{O}_{r_0}(\mu^{\frac{1}{2}} + \mu^{\frac{n}{2}}\mathcal{E}_{n}(\mu)^{-2}). 		
	\end{split}
\end{equation}
We note that the final remainder term in \eqref{match:SpotA;corefar-scale} can be estimated by $\mathcal{O}_{r_0}(\mu^{\kappa})$ for some $\kappa>0$ for each choice of $n>0$ and, in particular,
\begin{equation*}
    \mathcal{O}_{r_0}(\mu^{\frac{n}{2}}\mathcal{E}_{n}(\mu)^{-2}) = \begin{cases}
        \mathcal{O}_{r_0}(\mu^{\frac{n}{2}}), & 0<n<1,\\
        \mathcal{O}_{r_0}(\mu^{\frac{1}{2}}|\log(\mu)|), & n=1,\\
        \mathcal{O}_{r_0}(\mu^{\frac{1}{2}}), & n>1.\\
    \end{cases}
\end{equation*}
Then, taking $\mu,r_1,r_0^{-1}\ll1$, we obtain the leading-order system
\begin{equation*}
	\begin{split}
      \widetilde{a}\mathrm{e}^{\mathrm{i}\widetilde{Y}} &= \widetilde{d}_{1},\qquad
    		\widetilde{a}\mathrm{e}^{\mathrm{i}\widetilde{Y}} = \nu_n \widetilde{\gamma}\,\widetilde{d}_{1}^2 + \mathrm{i} \widetilde{d}_{2}. 
	\end{split}
\end{equation*}
which has a solution of the form
\begin{equation*}
    \widetilde{d}_1 = \widetilde{a} = \frac{1}{\nu_{n} \gamma}, \qquad \widetilde{d}_2=0,\qquad \widetilde{Y}=0.
\end{equation*}
\item In the case when $\gamma = \mu^{\frac{n}{4}}\mathcal{E}_{n}(\mu)^{-1} \tilde{\gamma}$, we instead introduce the following coordinate transformations 
\begin{equation}\label{match:SpotA;scaling-fold}
    	d_{1} = \sqrt{c_0}\mu^{\frac{2-n}{4}}\mathcal{E}_{n}(\mu)\widetilde{d}_{1},\qquad d_{2} = c_0\mu^{\frac{4-n}{4}}\mathcal{E}_{n}(\mu)\widetilde{d}_{2},\qquad a = \frac{\sqrt{c_0}}{2}\mathcal{E}_{n}(\mu)\widetilde{a}, \qquad Y = Y_{n} + \widetilde{Y},
\end{equation}
so that \eqref{match:SpotA;corefar} becomes  
\begin{equation}\label{match:SpotA;corefar-fold}
	\begin{split}
    		\widetilde{a}\mathrm{e}^{\mathrm{i}\widetilde{Y}}(1 + \xi_0) ={}& \widetilde{d}_{1} \left(1 + \mathcal{O}(r_0^{-1})\right) + \mathcal{O}_{r_0}(\mu^{\frac{1}{4}}),\\
    		\widetilde{a}\left[1+\frac{c_3\sqrt{c_0}}{4}\widetilde{a}^2 \right]  \mathrm{e}^{\mathrm{i}\widetilde{Y}}(1 + \xi_0) ={}& \tilde{\gamma} \,\widetilde{d}_{1}^2 \left(\nu_n  + \mathcal{O}\left(r_0^{-\frac{n}{2}}+r_0^{-1}\right)\right) + \mathrm{i} \widetilde{d}_{2} \left(1 + \mathcal{O}(r_0^{-1})\right) + \mathcal{O}_{r_0}(\mu^{\frac{1}{2}}). 
	\end{split}
\end{equation}
Note, here we have estimated $\mathcal{O}_{r_0}(\mu^{\frac{2-n}{4}}\mathcal{E}_{n}(\mu))$ as $\mathcal{O}_{r_0}(\mu^{\frac{1}{4}})$, since
\begin{equation*}
    \mathcal{O}_{r_0}(\mu^{\frac{2-n}{4}}\mathcal{E}_{n}(\mu)) = \begin{cases}
        \mathcal{O}_{r_0}(\mu^{\frac{2-n}{4}}), & 0<n<1,\\
        \mathcal{O}_{r_0}(\mu^{\frac{1}{4}}|\log(\mu)|^{-\frac{1}{2}}), & n=1,\\
        \mathcal{O}_{r_0}(\mu^{\frac{1}{4}}), & n>1.\\
    \end{cases}
\end{equation*}
Then, taking $0<\mu,r_1,r_{0}^{-1}\ll1$, we obtain the leading-order system
\begin{equation*}
	\begin{split}
    		\widetilde{a}\mathrm{e}^{\mathrm{i}\widetilde{Y}} &= \widetilde{d}_{1},\\
    		\widetilde{a}\left[1+\frac{\sqrt{c_0}c_3}{4}\widetilde{a}^2 \right]  \mathrm{e}^{\mathrm{i}\widetilde{Y}} &=\nu_n \widetilde{\gamma}\,\widetilde{d}_{1}^2 + \mathrm{i} \widetilde{d}_{2},
	\end{split}
\end{equation*}
which has solutions of the form
\begin{equation}
    \widetilde{d}_1 = \widetilde{a} = \frac{2\left[\nu_n \widetilde{\gamma}\,\pm\sqrt{\nu_n^2 \widetilde{\gamma}^2 - \sqrt{c_0}c_3}\right]}{\sqrt{c_0}c_3}, \qquad \widetilde{d}_2=0,\qquad \widetilde{Y}=0.
\end{equation}
\item Setting the far-field trajectory \eqref{Sol:Rings;r0} and the core parametrisation \eqref{Core:Am;param} equal to each other, we look to solve
\begin{equation}\label{match:Rings;corefar}
	\begin{split}
    		2q_n\mu^{\frac{4-n}{4}} r_0 \textnormal{e}^{\textnormal{i}(Y-Y_n)}(1 + \xi_1) &= d_1 \left(1 + \mathcal{O}(r_0^{-1})\right) -\mathrm{i}d_2  r_0\left(1 + \mathcal{O}(r_0^{-1})\right) + \mathcal{O}_{r_0}(|\mu||\mathbf{d}| + |\mathbf{d}|^{2}),\\
    		2 q_n\mu^{\frac{4-n}{4}}\textnormal{e}^{\textnormal{i}(Y-Y_n)}(1 + \xi_1) &= -\gamma\, d_1^2\left(\nu_n + \mathcal{O}\left(r_0^{-\frac{n}{2}}+r_0^{-1}\right)\right) - \mathrm{i} d_2 \left(1 + \mathcal{O}(r_0^{-1})\right)\\
      &\hspace{0.35\linewidth} + \mathcal{O}_{r_0}(|\mu||\mathbf{d}| + |d_2|^{2} + |d_1|^3). 
	\end{split}
\end{equation}
We introduce the following coordinate transformations 
\begin{equation}\label{match:Rings;scaling}
    	d_{1} = 2q_n\mu^{\frac{4-n}{4}}r_0\widetilde{d}_{1},\qquad d_{2} = 2q_n\mu^{\frac{4-n}{4}}\widetilde{d}_{2},\qquad Y = Y_{n} + \widetilde{Y},
\end{equation}
so that \eqref{match:Rings;corefar} becomes
\begin{equation}\label{match:Rings;corefar-scale}
	\begin{split}
    		\textnormal{e}^{\textnormal{i}\widetilde{Y}}(1 + \xi_1) &= \widetilde{d}_{1}\left(1 + \mathcal{O}(r_0^{-1})\right) -\mathrm{i}\widetilde{d}_{2}\left(1 + \mathcal{O}(r_0^{-1})\right) + \mathcal{O}_{r_0}(\mu^{\frac{4-n}{4}}),\\
    	\textnormal{e}^{\textnormal{i}\widetilde{Y}}(1 + \xi_1) &= - \mathrm{i} \widetilde{d}_{2} \left(1 + \mathcal{O}(r_0^{-1})\right) + \mathcal{O}_{r_0}(\mu^{\frac{4-n}{4}}). 
	\end{split}
\end{equation}
Then, taking $0 < \mu,r_{1}, r_{0}^{-1} \ll 1$, we obtain the leading-order system
\begin{equation*}
	\begin{split}
		\textnormal{e}^{\textnormal{i}\widetilde{Y}} &= \widetilde{d}_{1} -\mathrm{i}\widetilde{d}_{2},\qquad
    	\textnormal{e}^{\textnormal{i}\widetilde{Y}} = -\mathrm{i} \widetilde{d}_{2}. 
	\end{split}
\end{equation*}
 which has solutions of the form
\begin{equation*}
    \widetilde{d}_1 = 0, \qquad \widetilde{d}_{2} = \pm 1, \qquad \widetilde{Y} = \frac{(2\pm1)\pi}{2}.
\end{equation*}
\item Setting the far-field trajectory \eqref{Sol:SpotB;r0} and the core parametrisation \eqref{Core:Am;param} equal to each other, we look to solve
\begin{equation}\label{match:SpotB;corefar}
	\begin{split}
    		2a q_n\mu^{\frac{4-n}{8}} \textnormal{e}^{\textnormal{i}(Y-Y_n)}(1 + \xi_2) &= d_1 \left(1 + \mathcal{O}(r_0^{-1})\right) -\mathrm{i}d_2  r_0\left(1 + \mathcal{O}(r_0^{-1})\right) + \mathcal{O}_{r_0}(|\mu||\mathbf{d}| + |\mathbf{d}|^{2}),\\
    		2q_n\mu^{\frac{4-n}{4}}\textnormal{e}^{\textnormal{i}(Y-Y_n)}(1 + \xi_2) &= -\gamma\, d_1^2\left(\nu_n + \mathcal{O}\left(r_0^{-\frac{n}{2}}+r_0^{-1}\right)\right) - \mathrm{i} d_2 \left(1 + \mathcal{O}(r_0^{-1})\right)\\
      &\hspace{0.35\linewidth} + \mathcal{O}_{r_0}(|\mu||\mathbf{d}| + |d_2|^{2} + |d_1|^3). 
	\end{split}
\end{equation}
We introduce the following coordinate transformations 
\begin{equation}\label{match:SpotB;scaling}
    	d_{1} = \mu^{\frac{4-n}{8}}\widetilde{d}_{1},\qquad d_{2} = \mu^{\frac{4-n}{4}}\widetilde{d}_{2},\qquad Y = Y_{n} + \widetilde{Y},
\end{equation}
so that \eqref{match:SpotB;corefar} becomes
\begin{equation}\label{match:SpotB;corefar-scale}
	\begin{split}
    		2aq_n\textnormal{e}^{\textnormal{i}\widetilde{Y}}(1 + \xi_2) &= \widetilde{d}_{1}\left(1 + \mathcal{O}(r_0^{-1})\right) + \mathcal{O}_{r_0}(\mu^{\frac{4-n}{8}}),\\
    	2q_n\textnormal{e}^{\textnormal{i}\widetilde{Y}}(1 + \xi_2) &= -\gamma\, \widetilde{d}_1^2\left(\nu_n + \mathcal{O}\left(r_0^{-\frac{n}{2}}+r_0^{-1}\right)\right) -\mathrm{i}\widetilde{d}_{2}\left(1 + \mathcal{O}(r_0^{-1})\right)  + \mathcal{O}_{r_0}(\mu^{\frac{4-n}{4}}). 
	\end{split}
\end{equation}
Then, taking $0 < \mu,r_{1}, r_{0}^{-1} \ll 1$, we obtain the leading-order system
\begin{equation*}
	\begin{split}
		2aq_n\textnormal{e}^{\textnormal{i}\widetilde{Y}} = \widetilde{d}_{1},\qquad
    	2q_n\textnormal{e}^{\textnormal{i}\widetilde{Y}} = -\nu_n \gamma \widetilde{d}_1^2 -\mathrm{i}\widetilde{d}_{2}, 
	\end{split}
\end{equation*}
which has a solution of the form
\begin{equation*}
    \widetilde{d}_1 = -\mathrm{sgn}(\gamma)\sqrt{\frac{2 q_n}{\nu_n |\gamma|}}, \qquad \widetilde{d}_{2} = 0, \qquad a = \sqrt{\frac{1}{2q_n \nu_n |\gamma|}}, \qquad \widetilde{Y} = \frac{(\mathrm{sgn}(\gamma)-1)\pi}{2}.
\end{equation*}
    \end{enumerate}
    Having determined the leading order solution in each case, we now apply the implicit function theorem to solve \textnormal{(\ref{match:SpotA;corefar-scale}, \ref{match:SpotA;corefar-fold}, \ref{match:Rings;corefar-scale}, \ref{match:SpotB;corefar-scale})} uniquely for all $0<\mu, r_{1}, r_{0}^{-1}\ll1$. Inverting the respective coordinate transformations \textnormal{(\ref{match:SpotA;scaling}, \ref{match:SpotA;scaling-fold}, \ref{match:Rings;scaling}, \ref{match:SpotB;scaling})}, we arrive at our result.
\end{proof}

\section{Discussion}\label{s:discussion}

In this paper, we have proven that localised $(n+1)$-dimensional radial patterns emerge from non-degenerate Turing bifurcations in general two-component reaction-diffusion systems. Extending the approach of \cite{lloyd2009localized,mccalla2013spots,hill2023approximate,hill2024dihedral}, we determined that spot A patterns emerge for all values of the dimension parameter $n>0$, where the core profile depends continuously on $n$ and possesses a fixed asymptotic scaling $\mu^{\frac{1}{2}}$. In contrast, we found that ring and spot B patterns only emerge for $0<n<4$, where their core profiles depend continuously on $n$ and they possess an asymptotic scaling $\mu^{\frac{4-n}{4}}$, $\mu^{\frac{4-n}{8}}$, respectively. The core profile for each pattern depends on $(n+1)$-dimensional Bessel functions $J^{(n)}_{\ell}(r)$, which can be written in terms of standard Bessel functions; this highlights the connection between core profiles of two-dimensional and three-dimensional radial patterns that was previously unclear. We observe that spots are larger and sharper as $n$ increases, and that their oscillations are more strongly damped. This suggests that it may be harder to identify the associated wave number of localised three dimensional spots in physical and biological examples, preventing their characterisation as Turing patterns. However, our results show that such patterns are strongly associated with Turing bifurcations as well as their lower-dimensional counterparts. Our results agree with previous findings in \cite{lloyd2009localized,mccalla2013spots}, while providing a proof for the previously unproven existence of three-dimensional rings and presenting simplified proofs for spot B patterns and solutions of the non-autonomous Ginzburg--Landau equation \eqref{GL:comp} with $0<n<3$. 

We briefly discuss the intuition behind the asymptotic scaling $\mu^{\frac{4-n}{4}}$ for localised rings and the corresponding requirement that $0<n<4$. As mentioned previously, this scaling was predicted in \cite{mccalla2013spots} through a formal analysis, however the intuition for this scaling was not explored further. In the derivation for localised ring patterns $\mathbf{u}_{R}(r)$, we seek solutions that satisfy the far-field scaling $|\mathbf{u}_{R}(r)|=\mathcal{O}(\mu^{\frac{1}{2}})$ when $r=\mathcal{O}(\mu^{-\frac{1}{2}})$. If we suppose that our solution satisfies $|\mathbf{u}_{R}(r)| = \mathcal{O}(\mu^{\alpha}\,|r J_{1}^{(n)}(r)|)$ for some $\alpha>0$, then (using Table~\ref{table:Bessel}) $|\mathbf{u}_{R}(r)| = \mathcal{O}(\mu^{\alpha}\,r^{\frac{2-n}{2}})$ for large values of $r$. Furthermore, for $r=\mathcal{O}(\mu^{-\frac{1}{2}})$ we obtain $|\mathbf{u}_{R}(r)| = \mathcal{O}(\mu^{\alpha+\frac{n-2}{4}})$ and so rings must satisfy $\alpha + \frac{n-2}{4}= \frac{1}{2}$. Hence, the far-field profile is always smaller than $\mathcal{O}(\mu^{\frac{1}{2}})$ if $n>4$.
We note that localised rings may still emerge for $n>4$, except they will not satisfy the standard far-field scaling.
Instead, localised patterns with the form $|\mathbf{u}(r)| = \mathcal{O}(\mu^{\alpha}\,|r J_{1}^{(n)}(r)|)$ for $n>4$, either must satisfy $|\mathbf{u}(r)|=\mathcal{O}(\mu^{\kappa})$ with $\kappa > \frac{1}{2}$ in the region $r=\mathcal{O}(\mu^{-\frac{1}{2}})$,  or $|\mathbf{u}(r)|=\mathcal{O}(\mu^{\frac{1}{2}})$ in the region $r=\mathcal{O}(\mu^{-\beta})$ with $\beta<\frac{1}{2}$. 
We leave this approach for future study.

While we have considered the emergence of localised radial patterns from an $(n+1)$-dimensional Turing instability, very little is understood of their behaviour away from this bifurcation point. As solutions become larger, the nonlinearity of \eqref{R-DEqn} becomes more important and localised radial solutions may no longer reflect the profiles predicted by our analysis. As seen in Figure~\ref{fig:nSpot}(a), spot A patterns may undergo snaking behaviour as they move away from their bifurcation point. This is behaviour is known as \textit{homoclinic snaking} (first introduced in \cite{Woods1999Homoclinic}) when $n=0$, which has been the subject of over two decades of research, (see \cite{Champneys2021Editorial}) but is still not well-understood for $n>0$. While there has been some numerical studies of the bifurcation structure of spots in two and three dimensions \cite{McCalla2010Snaking,Gomila2021RadialSnaking,Sun2024dimensional}, and analytical results regarding ring patterns in $n$-spatial dimensions \cite{Bramburger2019Rolls}, this topic still requires further investigation.

A natural next step in this analysis would be to consider the stability of localised radial patterns for various spatial dimensions, for which there has been a recent numerical study \cite{Sun2024dimensional}. In particular, the transition of localised one-dimensional patterns to axisymmetric patterns (and from axisymmetric to spherically-symmetric) would be an interesting phenomenon to explore. A rigorous understanding of the stability of localised radial patterns is not currently well established, even in the simple cases of $n=1,2$, and so this would be a fruitful area of future study. Our current approach could also be extended to higher-dimensional symmetries, such as the dihedral patterns considered in \cite{hill2023approximate,hill2024dihedral}. Much of the approach presented \S\,\ref{s:Spatial_Dynamics} is equivalent to the results presented in \cite{hill2023approximate,hill2024dihedral}, and our simplified proof for Spot B patterns without blow-up coordinates presents a new strategy for obtaining dihedral spot B-type patterns in a future work. However, studying the dimensional effects of fully localised non-radial patterns would require some kind of formulation of $(n+1)$-dimensional spherical harmonics that vary continuously with respect to $n$ and would be a highly nontrivial extension of the work presented in \cite{hill2023approximate,hill2024dihedral}. 

\begin{Acknowledgment}
    The author gratefully acknowledges support from the Alexander von Humboldt foundation, and would like to thank David Lloyd for his comments on an earlier draft of this manuscript.   Additionally, the author would like to thank Boris Buffoni and François Genoud for their valuable insight on ground states to elliptic PDEs.
\end{Acknowledgment}

\bibliographystyle{abbrv}
\bibliography{Bibliography.bib}

\end{document}